\documentclass[11pt]{article}
\usepackage{amsmath, amssymb, amsthm}
\usepackage{graphicx,color}
\usepackage{multirow}
\usepackage{longtable}
\usepackage{url,geometry}
\usepackage{appendix}
\usepackage{framed}
\usepackage{tikz}
\usepackage{pdflscape}
\usetikzlibrary{decorations.pathreplacing}

\usepackage{hyperref}

\usepackage{natbib}
 \bibpunct[, ]{(}{)}{,}{a}{}{,}%

\newtheorem{proposition}{Proposition}

\newtheorem{lemma}{Lemma}
\theoremstyle{definition}
\newtheorem{definition}{Definition}

\newtheorem{remark}{Remark}

\newcommand{\exclude}[1]{}
\newcommand{\Ze}{\mathbb{Z}}

\renewcommand{\vec}[1]{\mathbf{#1}}

\newcommand{\PL}{P}

\newcommand{\sd}{\text{speed}}
\newcommand{\dpt}{\text{depot}}

\title{A Joint Routing and Speed Optimization Problem}
\author{Ricardo Fukasawa\thanks{Department of Combinatorics and Optimization, Faculty of Mathematics, University of Waterloo. Email: rfukasawa@uwaterloo.ca.}, Qie He\thanks{Department of Industrial and Systems Engineering, University of Minnesota. Email: qhe@umn.edu.}, Fernando Santos\thanks{Department of Engineering, Federal University of Itajub\'a - Campus Itabira. Email: fernandosantos@unifei.edu.br}, Yongjia Song\thanks{Department of Statistical Sciences and Operations Research, Virginia Commonwealth University. Email: ysong3@vcu.edu.}}
\date{}

\begin{document}
\maketitle
\begin{abstract}
Fuel cost contributes to a significant portion of operating cost in cargo transportation. Though classic 
routing models usually treat fuel cost as input data, fuel consumption heavily depends on the travel speed, 
which has led to the study of optimizing speeds over a given fixed route.
In this paper, we propose a joint routing and speed optimization problem to minimize the total cost, which includes the fuel consumption cost.
The only assumption made on the dependence between the fuel cost and travel speed is that it is a strictly convex differentiable function. 
This problem is very challenging, with medium-sized instances already difficult for a general mixed-integer convex optimization solver. 
We propose a novel set partitioning formulation and a branch-cut-and-price algorithm to solve this problem. 
Our algorithm clearly outperforms the off-the-shelf optimization solver, and is able to solve some 
benchmark instances to optimality for the first time.
\end{abstract}


\section{Introduction}
Speed has significant economic and environmental impacts on transportation. Slow streaming, meaning that a ship travels at a much lower speed than its maximum speed, is a common practice in maritime transportation to reduce fuel consumption and greenhouse gas emissions~\citep{maloni2013slow}. For many diesel-powered heavy-duty trucks, the fuel consumption rate is approximately a cubic function of its traveling speed~\citep{barth2004modal}, implying that a $10\%$ reduction in speed could save approximately $27\%$ fuel consumption. In addition, greenhouse gas (GHG) emissions, especially CO2 emissions, are directly proportional to the fuel consumption, so improved fuel efficiency also implies less GHG emissions. On the other hand, a lower speed makes it more difficult to satisfy customers' service requests, which usually are required to be done within certain time windows, and therefore affects other operational decisions such as routing and fleet deployment. 

In this paper, we propose a transportation model that jointly optimizes the routing and speed decisions with an objective of minimizing total costs, including fuel consumption costs. In particular, we are interested in the following joint routing and speed optimization problem (JRSP).
\begin{framed}
Given a network of customer requests, how can we find the set of routes and at what speed should each vehicle/ship travel over each leg of the route to minimize the overall cost, while respecting all the operational constraints such as capacity and service time windows?
\end{framed}

We model the fuel cost over each leg as a strictly convex differentiable function of the average travel speed over that leg. This assumption is satisfied by many empirical fuel consumption and emission models in maritime and road transportation, such as the cubic approximation used for tankers, bulk carriers, or ships of small size~\citep{notteboom2009effect}, the transportation emissions and energy consumption models developed by the European Commission~\citep{hickman1999methodology}, and the comprehensive emission model for trucks~\citep{barth2000development}.

The combination of routing and speed decisions makes the JRSP significantly more difficult to solve than the more classical vehicle routing problems (VRPs). For instance, there are JRSP instances of 20 customers that cannot be solved to optimality by state-of-the-art optimization software, while the latest exact algorithm for capacitated vehicle routing problems is able to routinely solve instances with hundreds of customers~\citep{pecin2014improved}. The reason is that JRSP is essentially a mixed-integer convex (non-linear) program (MICP), and the development of general MICP solvers is significantly lagging behind that of mixed-integer linear program solvers. 
The goal of this work is to reduce this performance discrepancy by developing an efficient algorithm for the JRSP 
that exploits its structure.

The most successful exact algorithms in practice for classical VRPs are based on the 
branch-and-price algorithm~\citep{Barnhart1998BP, Desrochers1992VRPTW, Lubbcke2005CG}, in which a set
covering/partitioning formulation is solved by branch and bound, the linear programming relaxation at 
each node of the branch-and-bound tree is solved by column generation, and the problem of generating a column 
of the linear program with negative reduced cost, namely the \emph{pricing} problem, is solved by an efficiently
implemented labeling algorithm. Cutting planes can also be added at each node of the branch-and-bound tree
to further strengthen the relaxation, leading to the branch-cut-and-price (BCP) algorithm.

One of the most 
critical components that affect the computational efficiency of these algorithms is how the pricing problem
is solved. The pricing problem of the VRP is typically formulated as a (elementary or non-elementary) shortest path problem with resource constraints (SPPRC)~\citep{Irnich2005SPPRC}. The elementary version is known to be strongly 
NP-hard~\citep{Dror1994complexity}, and specialized dynamic programming algorithms and several fast heuristics are used to solve it within reasonable time limits. 
In classical VRPs, the cost of a given route in the pricing problem 
is easy to compute by simply keeping track of the cumulative cost incurred so
far and adding the cost of any extra arc.
The additional challenge of the JRSP is that
the speed over each arc (and thus the cost) is no longer a resource that gets accumulated throughout the route.
In fact it is a decision variable that affects both feasibility and optimality of the routes and so the optimal 
speed (and ultimately the route's cost) can only be determined after the whole
route is given. There is no clear notion of a partial cost that gets accumulated as
the partial route is determined. For instance, a partial route may have its
lowest cost by being traversed in a very low speed, but to be able to extend it to a
new customer in time, the speeds of the partial route need to be significantly
increased and so the final cost cannot be determined based on combining partial
routes.



We propose a novel set partitioning formulation and a BCP algorithm to overcome such a challenge. 
In our formulation, each column represents the combination of a route and a speed profile over that route. 
The pricing problem, which can be seen as a joint shortest path and speed optimization problem, 
can be solved efficiently by a labeling algorithm due to our new formulation. 
The novelty of our labeling algorithm is that, for every path extension, we generate \emph{three} 
labels for each customer to be visited next: two labels in which this customer is served at the beginning 
or the end of its time window, and one label in which we only store a speed vector that satisfies 
the property of optimal speeds. With this new formulation, we do not need to keep track of every 
possible speed profile when generating a route, and are able to derive easy-to-check dominance rules 
that keep the pricing problem manageable. We test our algorithm on a variety of instances in maritime 
and road transportation. Our algorithm shows a significant advantage over one state-of-the-art general optimization solver. The contributions of this paper can be summarized as follows.
\begin{itemize}
	\item We propose a general joint routing and speed optimization model to improve fuel efficiency and reduce emissions. The model is able to accommodate many fuel consumption and GHG emission models used in practice.
	\item We propose a new set partitioning formulation for the JRSP by exploring the property of optimal speeds over a fixed route. We then develop a practical BCP algorithm in which the pricing problem can be solved efficiently by a labeling algorithm with easy-to-check dominance rules.
	\item We test on a set of instances in maritime and road transportation. We are able to solve instances of much larger sizes than the ones that a state-of-the-art optimization software could handle. To the best of our knowledge, some instances in the literature are solved to optimality for the first time.
\end{itemize}

The rest of the paper is organized as follows. Section~\ref{sec:literature} reviews models and algorithms related to the JRSP. Section~\ref{sec:problem} gives a detailed description of the problem and introduces the new set partitioning formulation. Sections~\ref{sec:labeling} and~\ref{sec:dominance} elaborate the labeling algorithm for the pricing problem as well as the dominance rules used in the labeling algorithms. Section~\ref{sec:other} introduces other components of the proposed algorithm. Extensive computational results are provided in Section~\ref{sec:computation}. We conclude in Section~\ref{sec:conclusion}.

\section{Literature Review} \label{sec:literature}

\subsection{The branch-and-price algorithm for the VRP}
The branch-and-price and BCP algorithms have consistently been the most successful exact solution methods for a large variety of vehicle routing and crew scheduling problems in practice~\citep{Barnhart1998BP, Desrochers1992VRPTW, pecin2014improved}, due to the modeling power of the set partitioning/covering formulation in handling complex costs and side constraints~\citep{desaulniers1998unified}, the tight relaxation bounds obtained by these formulations, and effective algorithms that handle the dynamic generation of columns in solving the relaxation~\citep{desaulniers2006column, Lubbcke2005CG}, among others. For the VRP, the pricing problem of column generation is usually formulated as an elementary shortest path problem with resource constraints. This problem is strongly NP-hard~\citep{Dror1994complexity}, and to solve it efficiently in practice, problem-specific dominance rules are needed to avoid enumerating too many path extensions. In addition, the requirement of an elementary route is usually relaxed, allowing routes with cycles~\citep{Desrochers1992VRPTW} or ng-routes~\citep{baldacci2011ng}, to speed up the column generation procedure.

\subsection{The speed optimization problem}
Given a fixed route, the problem of finding the optimal speed over each arc to minimize the total fuel consumption while respecting time-window constraints is called the speed optimization problem (SOP). The SOP is first considered by~\cite{fagerholt2010reducing} in the context of maritime transportation. When the fuel consumption function is convex in the average speed over each arc, the SOP is a convex minimization problem. In \cite{norstad2011tramp} an iterative algorithm to solve the SOP is developed. The algorithm runs in time quadratic in the number of vertices on the route, and its correctness is proven in~\cite{hvattum2013analysis} assuming the fuel consumption function is convex and non-decreasing in the average speed. The work in \cite{kramer2014matheuristic} further modifies the algorithm in~\cite{norstad2011tramp} to solve the SOP with labor cost and varying departure time at the depot. The correctness of these adapted algorithms in~\cite{demir2012adaptive} and \cite{kramer2014matheuristic} is proved recently in~\cite{kramer2015speed}.

\subsection{Integration of speed with other decisions}
In maritime transportation, the integration of sailing speeds with other decisions such as liner 
dispatching and routing~\citep{wang2012sailing, wang2016fundamental} and berth 
allocation~\citep{alvarez2010methodology, du2011berth} have been studied extensively; we refer the readers to 
recent surveys~\citep{meng2013containership, psaraftis2014ship} and the references therein. 
One problem closely related to our problem is the tramp ship routing and scheduling
problem~\citep{norstad2011tramp}, in which each tramp ship needs to select the pickup and delivery
location, its route, and the sailing speed over each leg of the routes to maximize the overall operating profit.
The work in \cite{norstad2011tramp} proposes a heuristic for this problem.

In road transportation, many recent VRPs have considered minimizing fuel consumption and emissions as the
main objective~\citep{demir2014review,eglese2014,lin2014survey, fukasawa2014branch, bektacs2016green}, among 
which only a few treat speed as a decision variable. The pollution routing problem (PRP), proposed
by~\cite{bektacs2011pollution}, aims to find a set of routes as well as speeds over each arc of the routes
to minimize the total fuel, emission, and labor costs. The fuel consumption function used in the PRP is based
on the emission model developed in~\cite{barth2004modal} for a heavy-duty diesel truck. 
The PRP is
solved by a mixed-integer linear program solver in~\cite{bektacs2011pollution} by allowing ten speed values
over each arc, and later by an adaptive large neighborhood search heuristic in~\cite{demir2012adaptive}.
The work in \cite{kramer2014matheuristic} develops a multi-start iterative local search framework for the
PRP with varying departure time. In \cite{jabali2012analysis}, a time-dependent VRP is studied, in which the whole planning horizon is divided
into two periods: a peak period with fixed low vehicle speed and an off-peak period with high free-flow speed, 
and the goal is to find a set of routes and a uniform upper bound on the free-flow speed over all arcs to 
minimize the total fuel, emission, and labor costs. The model is solved by a tabu search heuristic. 

Motivated by the integrated models in maritime transportation and the PRP, we propose the JRSP with general strictly convex cost functions, accommodating many fuel consumption and emission models \citep{hickman1999methodology, barth2000development, notteboom2009effect, fagerholt2010reducing, psaraftis2013speed}. 
Two algorithmic results closely related to our work are \cite{dabia2014exact} and \cite{Fuka2015TS}. In \cite{dabia2014exact}, a branch-and-price algorithm is developed for a variant of the PRP which assumes that the speed must be constant throughout every arc of a given route. In \cite{Fuka2015TS}, an arc-based mixed-integer second-order cone formulation is proposed for the PRP. Therefore the work proposed in this article is, to the best of our knowledge, the first exact
branch-and-price based approach for the JRSP. The results show that our new formulation and algorithm indeed allow significant reduction in terms of computational time.

\section{Problem Description and Formulation} \label{sec:problem}
\subsection{Problem description}
Throughout the paper, we work on the delivery version of the JRSP with asymmetric distance matrix and homogeneous vehicles. The pickup version can be handled in a similar way. Let $N = (V,A)$ be a complete directed graph and $V=\{0,1,2,\ldots,n\}$, where vertex $0$ denotes the depot and $V_0:=\{1,2,\ldots,n\}$ denotes the set of customers. There are $K$ identical vehicles, each with capacity $Q$, available at the depot. Each customer $i\in V_0$ has a demand $q_i$, a service time window $[a_i,b_i]$, within which the customer can be served, and a service duration $\tau_i$. The depot has a time window $[a_0, b_0]$ which represents the whole planning horizon. For each $(i,j)\in A$, $d_{ij}$ is the distance from vertex $i$ to vertex $j$. The fuel consumption per unit distance traveled (liter per kilometer) at speed $v$ is $f(v)$. We assume that the rate function $f$ is \emph{strictly convex and differentiable}. These properties are satisfied by many empirical fuel consumption models in maritime transportation~\citep{notteboom2009effect, psaraftis2013speed} and road transportation~\citep{hickman1999methodology, barth2000development, barth2004modal}, which are typically polynomial functions. Then the cost over arc $(i,j)$ is $d_{ij}f(v_{ij})$. We also assume that the vehicle speed over each arc cannot exceed a speed limit $u$ that is common for all arcs.
\begin{definition}
A \emph{route} is a walk $(i_0, i_1,\ldots,i_h,i_{h+1})$, where $i_0=i_{h+1}=0$, $i_j\in V_0$ for $j=1,\ldots,h$. An {\em elementary route} is a feasible route, i.e., a route that satisfies all the time window constraints as well as the vehicle capacity constraint, in which no customer is visited more than once and the total demand does not exceed the vehicle capacity. A \emph{q-route} is a feasible route over which customers are allowed to be visited more than once. A $\emph{$2$-cycle-free q-route}$ is a q-route that forbids cycles in the form of $(i, j, i)$ for $i,j \in V_0$.
\end{definition}
\noindent The goal of the JRSP is to select a set of elementary routes for all the vehicles and speeds over each arc of the elementary routes to minimize the total cost and respect the following constraints:
\begin{enumerate}
	\item 
	Each customer is served within its time window; 
 \item The total demands on each route do not exceed the vehicle capacity $Q$;
	\item 
The speed over each arc does not exceed the speed limit $u$.
\end{enumerate}
The cost of route $r$ with a speed vector $\vec{v}$, $c_r(\vec{v})$, is the summation of costs of all arcs over $r$. In particular, given a route $r=(i_0,i_1,\ldots,i_h,i_{h+1})$ and a speed vector $\vec{v}=(v_{i_{k-1}, i_{k}})_{k=1}^{h+1}$ over $r$, the route cost is calculated as follows:
\begin{equation} \label{eq:routecost}
c_r(\vec{v}) = \sum_{k=1}^{h+1} d_{i_{k-1}, i_k} f(v_{i_{k-1},i_k}).
\end{equation}

\subsection{A basic formulation for the JRSP}
We first introduce a naive set-partitioning formulation for the JRSP. Let $z_r$ be a binary variable indicating whether or not route $r$ is included in the optimal solution, and let $\alpha_{ir}$ be a parameter that denotes the number of times route $r$ visits customer $i\in V_0$. Then the JRSP can be formulated as follows:
\begin{subequations} \label{eq:sp0}
\begin{align} 
\min \ & \sum_{r\in \Omega_1} c_r z_r\\ 
\text{s.t.} \ & \sum_{r\in \Omega_1}\alpha_{ir}z_r = 1, \; i \in V_0 \\
& \sum_{r\in \Omega_1} z_r = K, \\
& z_r \in \{0, 1\}, \; r\in \Omega_1,
\end{align}
\end{subequations}
where $\Omega_1$ is the set of feasible routes and $c_r$ is the ``cost'' of route $r$. A caveat of this formulation is that the cost of a route $r$ also depends on the speed vector over the route, so $c_r$ is in fact not well defined given only the route information. To resolve this issue, we can define $c_r$ to be the minimum route cost over all feasible speed vectors. In particular, given a route $r=(i_0,i_1,\ldots,i_h,i_{h+1})$ with $i_0=i_{h+1}=0$, let variables $t_{i_l}$ denote the service start time at customer $i_l$ for $l=1,\ldots, h$ and variables $t_{i_0}$ and $t_{i_{h+1}}$ denote the time that the vehicle departs from and returns to the depot, respectively. Then $c_{r} $ is computed as follows.  
\begin{equation} \label{eq:cr}
\begin{split}
c_{r}\; = \; &\min c_r(\vec{v}) \\
\text{s.t.} \qquad &  t_{i_l} \ge t_{i_{l-1}} + \tau_{i_{l-1}} + \frac{d_{i_{l-1},i_l}}{v_{i_{l-1},i_l}}, l=1,\ldots, h,\\
& t_{i_{h+1}} \ge t_{i_h} + \tau_{i_h} + \frac{d_{i_h,0}}{v_{i_h,0}},\\
&a_{i_l} \le t_{i_l} \le b_{i_l}, l=1,\ldots,h+1,\\
& t_{i_0} = 0,\\
& v_F \le v_{i_{l-1}, i_l} \le u, l=1,\ldots,h+1,
\end{split}
\end{equation}
where the form of $c_r(\vec{v})$ is given in~\eqref{eq:routecost}, $a_{i_{h+1}}=a_0$, $b_{i_{h+1}}=b_0$, and $v_F$ is the speed that minimizes the fuel consumption rate function $f(v)$. Note that we made the assumption that the vehicle will never travel at a speed lower than $v_F$ without loss of generality, since otherwise the vehicle can instead travel at speed $v_F$ and wait, incurring a lower cost. 


\subsection{The proposed set partitioning formulation}
It is not difficult to observe that the pricing problem of~\eqref{eq:sp0} will be challenging to solve, since computing $c_r$ for a given route $r$ is already a non-trivial problem. To circumvent this difficulty, we develop a novel set partitioning formulation, in which each column represents a combination of a route and a set of speed values over the arcs of the route. This formulation contains much more columns than the naive set partitioning formulation~\eqref{eq:sp0}, but its advantage is that the corresponding cost coefficient in the objective for each column can be computed easily by solving several one-dimensional convex optimization problems. This is critical for an efficient labeling algorithm for the pricing problem.

To introduce this new set partitioning formulation, we first introduce the concept of an \emph{active} customer based on its actual service start time.
\begin{definition}
Given a route with a speed vector on the route, a customer $i_j$ is \emph{active}, if it is served at its earliest or latest available time, i.e., $a_{i_j}$ or $b_{i_j}$; we call $j$ an \emph{active index} on the route. Otherwise we call this customer a \emph{non-active} customer.
\end{definition}

\noindent Then we can characterize the optimal speeds over a given route $r$ with the following proposition.
\begin{proposition} \label{prop:sop:speed}
Any optimal speed vector for the speed optimization problem~\eqref{eq:cr} over a given route $r$ satisfies the following property: the speed is uniform between any two consecutive active customers, i.e., for any non-active customer, the speed on the arc entering the customer is equal to the speed on the arc emanating from that customer.
\end{proposition}
\begin{proof}
We prove the statement by contradiction. Suppose the speed optimization problem~\eqref{eq:cr} has an optimal speed vector $\vec{v}$ with three consecutive vertices $j,i,k$ such that customer $i$ is non-active ($t_i \in (a_i,b_i)$) and $v_{ji} \neq v_{ik}$. We show that we can create a new feasible speed vector with a strictly lower cost than the optimal speed vector. The new speed vector differs from the optimal speed vector only by the speeds on arcs $(j,i)$ and $(i,k)$.

Let the departure time at vertex $j$ be $T_j$ and the arriving time at vertex $k$ be $T_k$. We first assume that $v_{ji} < v_{ik} \le u$. The cost over arcs $(j,i)$ and $(i,k)$ is $d_{ji}f(v_{ji}) + d_{ik}f(v_{ik})$. We create two new speeds $v'_{ji}= v_{ji} +  d_{ik} \delta$ and $v'_{ik} = v_{ik} -  d_{ji} \delta$ with
\[ 0 < \delta <  \min \{\frac{v_{ik}-v_{ji}}{d_{ik}}, \frac{v_{ik}-v_{ji}}{d_{ji}}, \frac{v_{ik}^2 - v_{ji}^2}{v_{ji} d_{ik} + v_{ik} d_{ji}}\}.\]
We will show that with these two new speeds, the vehicle is able to depart vertex $j$ at time $T_j$, serve customer $i$ within its time window, arrive at node $k$ at time $T_k$, and incur a lower cost over arcs $(j,i)$ and $(i,k)$.

According to the choice of $\delta$, we have $v_{ji}', v_{ik}' \in (v_{ji}, v_{ik})$, so $v_{ji}'$ and $v_{ik}'$ are feasible speeds. Since the function $f(v)$ is strictly convex in $v$, then
\[\frac{f(v'_{ji}) - f(v_{ji})}{ d_{ik} \delta} =\frac{f(v_{ji}+d_{ik}\delta) - f(v_{ji})}{ d_{ik} \delta} <  \frac{f(v_{ik}) - f(v_{ik}-d_{ji}\delta)}{ d_{ji} \delta}=\frac{f(v_{ik}) - f(v'_{ik})}{ d_{ji} \delta}.\]

\noindent Thus $d_{ji}f(v'_{ji})  + d_{ik} f(v'_{ik})  <  d_{ji}f(v_{ji}) + d_{ik}f(v_{ik})$, which implies the new speeds incur less cost on the route. Now we show that with the new speeds, the vehicle is able to leave vertex $j$ at $T_j$ and arrive at vertex $k$ at time $T_k$. The original travel time over arcs $(j,i)$ and $(i,k)$ is
\[T =\frac{d_{ji}}{v_{ji}} + \frac{d_{ik}}{v_{ik}},\]
and the travel time with the new speeds is
\[T' =\frac{d_{ji}}{v_{ji} + d_{ik} \delta} + \frac{d_{ik}}{v_{ik} - d_{ji}\delta}.\]
Then
\begin{equation*}
\begin{split}
T' - T & =\frac{d_{ji}}{v_{ji} +  d_{ik}\delta} + \frac{d_{ik}}{v_{ik} -  d_{ji} \delta} -(\frac{d_{ji}}{v_{ji}} + \frac{d_{ik}}{v_{ik}}) \\
& =\frac{d_{ji}d_{ik}\delta [(v_{ji}d_{ik}+v_{ik}d_{ji}) \delta-(v_{ik}^2- v_{ji}^2)]}{v_{ik}v_{ji} (v_{ik}- d_{ji} \delta)(v_{ji} + d_{ik}\delta)} \\
& < 0.
\end{split}
\end{equation*}
The last inequality follows from the choice of $\delta$. Since the total travel time between vertex $j$ and vertex $k$ decreases with the new speeds, with sufficiently small $\delta$ the vehicle is able to satisfy the time window constraints for all customers while incurring a lower cost. The case with $v_{ji} > v_{ik}$ can be proven in a similar way. 
\end{proof}

\begin{remark}
According to Proposition~\ref{prop:sop:speed}, if we know two consecutive active customers on a given route and their service start times, then the optimal speeds over the arcs between the two customers can be calculated by searching one speed value such that the convex objective function is minimized and service time windows are respected for customers in between. This is essentially a \emph{one-dimensional convex optimization problem}. In other words, once we know two consecutive active customers, we can easily calculate the optimal speed over this segment of the route between the two customers, \emph{without the need of knowing how the whole route looks like}. This nice property motivates the new set partitioning formulation below, and forms the foundation of our efficient labeling algorithm in Section~\ref{sec:labeling} for the pricing problem.
\end{remark}

We now introduce the new set partitioning formulation. In this formulation, each column represents a triple $(r, I, \vec{s})$, where $r=(i_0,i_1,\ldots,i_h,i_{h+1})$ is a candidate route, $I$ is the set of active indices on route $r$, and $\vec{s} = (s_{j})_{j\in I}$ is a vector of service start times at active customers with $s_j\in \{a_{i_j},b_{i_j}\}$ for any $j\in I$. The cost $c_{r,I, \vec{s}}$ of a column in the formulation is defined as the minimum cost of route $r$ over any speed vector that guarantees the pattern $(I, \vec{s})$, i.e., customer $i_j$ is served at $s_j$ for all $j\in I$ with that speed vector. In particular, with the same notations in~\eqref{eq:cr}, the cost $c_{r,I,\vec{s}} $ is computed as follows.  
\begin{equation} \label{eq:optcost}
\begin{split}
c_{r,I,\vec{s}}\; = \; &\min c_r(\vec{v}) \\
\text{s.t.} \qquad &  t_{i_l} \ge t_{i_{l-1}} + \tau_{i_{l-1}} + \frac{d_{i_{l-1},i_l}}{v_{i_{l-1},i_l}}, l=1,\ldots, h,\\
& t_{i_{h+1}} \ge t_{i_h} + \tau_{i_h} + \frac{d_{i_h,0}}{v_{i_h,0}},\\
& a_{i_l} \le t_{i_l} \le b_{i_l}, l \notin I,\\
& t_{i_0} = 0, \\
& t_{i_j} = s_j, j\in I,\\
& v_F \le v_{i_{l-1}, i_l} \le u, l=1,\ldots,h+1.
\end{split}
\end{equation}
We set $c_{r,I,\vec{s}}= \infty$ if the cumulative demands on the route exceed the vehicle capacity or the optimization problem in~\eqref{eq:optcost} is infeasible. The difference between problem~\eqref{eq:cr} and problem~\eqref{eq:optcost} is that in~\eqref{eq:optcost} variables $t_{i_j}$ for $j\in I$ have fixed values, so~\eqref{eq:optcost} can be seen as a restriction of~\eqref{eq:cr}. According to Proposition~\ref{prop:sop:speed}, problem~\eqref{eq:optcost} can be decomposed into $|I|+1$ one-dimensional convex optimization problems.

Let $R$ be the set of all the routes in the network. Let $\Omega_2$ be the set of all such triples $(r,I,\vec{s})$, i.e., $\Omega_2=\{(r,I,\vec{s}) \mid I \text{ is a subset of indices of route } r, r\in R\}$. Note that $\Omega_2$ is a finite set. Let a binary variable $z_{r, I, \vec{s}}$ indicate whether or not pattern $(r, I, \vec{s})$, i.e., a route $r$ with active indices $I$ and the corresponding service start time vector $\vec{s}$, is included in the optimal solution. Our set partitioning formulation of the JRSP is stated as follows.
\begin{subequations} \label{eq:sp}
\begin{align}
\min \ & \sum_{r\in \Omega_2} c_{r,I,\vec{s}} z_{r,I,\vec{s}} \\
\text{s.t.} \ & \sum_{(r,I,\vec{s})\in \Omega_2} \alpha_{ir} z_{r,I,\vec{s}} =1, \; i\in V_0, \label{eq:sp:indegree}\\
 \ & \sum_{(r,I,\vec{s})\in \Omega_2} z_{r,I,\vec{s}} = K, \label{eq:sp:cardinality}\\
\ & z_{r,I,\vec{s}} \in \Ze_+, \; (r,I,\vec{s})\in \Omega_2.
\end{align}
\end{subequations}
It is clear that~\eqref{eq:sp} gives a valid formulation for the JRSP, since all possible sets of active customers are considered on any given route and only elementary routes give integer feasible solutions. We can also replace routes $r$ in the formulation by elementary routes,
$ng$-routes, $k$-cycle-free $q$-routes, or any set of routes that includes elementary routes, and formulation~\eqref{eq:sp} remains valid.

We solve formulation~\eqref{eq:sp} using a branch-and-bound algorithm, in which the linear programming relaxation at each node of the branch-and-bound tree is solved by column generation. The column generation master problem is a restricted linear program consisting of only a subset of columns in $\Omega_2$. Let the dual variables corresponding to constraints~\eqref{eq:sp:indegree} and~\eqref{eq:sp:cardinality} be $\mu_i$ for $i\in V_0$ and $\nu$, respectively, the \emph{pricing problem} for generating a new column to be added to the master problem is formulated as follows:
\begin{equation} \label{eq:pricing}
\min_{(r,I,\vec{s})\in \Omega_2} \bar{c}_{r,I,\vec{s}}:= c_{r,I,\vec{s}}-\sum_{i\in V_0}\mu_i a_{ir} - \nu.
\end{equation}
In the next section, we introduce an efficient labeling algorithm to solve the pricing problem~\eqref{eq:pricing}.

\section{The Labeling Algorithm} \label{sec:labeling}
We solve the pricing problem \eqref{eq:pricing} by a labeling algorithm, in which the optimal triple $(r, I, \vec{s})$ is generated dynamically. The general idea is as follows. When we extend a walk to a new customer $j$, it could be active or non-active. If customer $j$ is active, we add its index to set $I$; the optimal speed vector over the walk up to customer $j$ can be computed, regardless of how the walk is extended after customer $j$. If customer $j$ is non-active, let $i_w$ be the last known active customer on the walk. Then the optimal speeds over the arcs between customer $i_w$ and customer $j$ must be the same, according to Proposition~\ref{prop:sop:speed}, so that we can store the total fuel cost consumed between customer $i_w$ and customer $j$ using a one-dimensional convex function of the uniform speed between $i_w$ and $j$. This makes it very easy to compare two labels as we will show in Section~\ref{sec:dominance}.

Before formalizing the idea, we need to address a technical issue on whether the optimal speed of the one-dimensional convex optimization problem can be attained. Since the service start time at a non-active customer $j$ lies in an open interval, the set of speed values that makes $j$ non-active will be an open set. Then the optimal speed that minimizes the one-dimensional convex fuel cost function may not be attained. To address this issue, we introduce the concept of a \emph{seamless customer} to make the set of feasible speed values a closed set.
\begin{definition} \label{def:seamless}
Given a route with a speed vector on the route, a customer $j$ is \emph{seamless}, if the following two conditions hold: (1) The speed on the arc entering $j$ is equal to the speed on the arc emanating $j$; (2) There is no waiting at customer $j$.
\end{definition}

\noindent Given a route with an associated speed vector, the set of seamless customers contains the set of non-active customers. To see this, any non-active customer on a route with an optimal speed vector must be a seamless customer by Proposition~\ref{prop:sop:speed}. But the converse is not true in general, since a seamless customer is allowed to be served at the boundary point of its time window, i.e., being active. By introducing the concept of a seamless customer, every optimization problem arising in the pricing problem will be defined on a closed set and its optimal solution will be attained.

\subsection{Label definition}
We first define some notations related to a walk. Given a walk $P=(i_0, i_1,\ldots, i_h)$, let $q(P)=\sum_{k=1}^hq_{i_{k-1},i_k}$ be the cumulative demands delivered on $P$, $\tau_{i_l, i_m}(P) =\sum_{k=l}^m\tau_{i_k}$ be the total service time spent between customers $i_l$ and $i_m$ on $P$ (including customers $i_l$ and $i_m$), $D_{i_l,i_m}(P)=\sum_{k=l+1}^m d_{i_{k-1},i_k}$ be the total distance between customer $i_l$ and customer $i_m$, for $0\le l\le m \le h$. We write $\tau_{i_l, i_m}$ and $D_{i_l, i_m}$ instead of $\tau_{i_l, i_m}(P)$ and $D_{i_l, i_m}(P)$ when the context is clear.

We develop a forward labeling algorithm to solve the pricing problem~\eqref{eq:pricing}. A label $L = (\PL,w,s)$ is associated with: (i) a walk $\PL=(i_0,i_1,\ldots,i_h)$ with $i_0=0$ and $q(P)\le Q$; (ii) the index $w$ of the \emph{last known} active customer on $\PL$; (iii) the service start time $s$ of customer $i_w$, which is either $a_{i_w}$ or $b_{i_w}$.
Note that we only need to store the last known active index instead of all active indices, since the walk $P$ is generated dynamically. The optimal speeds over arcs before $i_w$ will be computed and stored during the generation of the walk. A label $L = (\PL,w,s)$ contains the following attributes:
\begin{itemize}
\item $M$, the set of forbidden vertices, i.e., vertices that cannot be visited directly after vertex $i_h$ along $\PL$. For example, if we consider only elementary routes in our set-partitioning formulation, then $M$ is the set of vertices in $P$; if we consider q-routes, then $M=\{i_h\}$.
\item $\mu(P)$, the sum of dual multipliers along $P$, i.e., $\mu(P)=\sum_{k=1}^h{\mu_{i_k}}$.
\item $q(P)$, the cumulative demands delivered along $P$, i.e., $q(P)=\sum_{k=1}^h{q_{i_k}}$.
\item $S_v$, a set of feasible speeds $v$ such that the time-window constraint of each customer between $i_w$ and vertex $i_h$ along $P$ is satisfied, by traveling at speed $v$ between $i_w$ and $i_h$ without any waiting in between. In particular,
\[S_v=\left\{v\in [l,u] \mid s+\tau_{i_w,i_{j-1}}+ \frac{D_{i_w,i_j}}{v} \in [a_{i_j},b_{i_j}], \forall j=w+1,\ldots,h\right\}.\]
\item $\Gamma$, the total service time spent on customers from $i_w$ to $i_h$ along $\PL$, i.e., $\Gamma =\tau_{i_w, i_h}$.
\item $D$, the total distance between vertex $i_w$ and vertex $i_h$ along $\PL$, i.e., $D=D_{i_w,i_h}$.
\item $F_{\sd}$, the optimal cost on arcs before customer $i_w$ along $\PL$. This quantity is dynamically updated when the last known active customer along $\PL$ is updated. The initialization and update of $F_{\sd}$ is explained in Section~\ref{subsec:labelext}.

\end{itemize}
The attributes $S_v$, $\Gamma$, $D$, and $F_{\sd}$ are different from the resources as in the resource constrained shortest path problem. 


\subsection{Label initialization and extension} \label{subsec:labelext}
From this point on, we will refer to labels and their attributes with superscripts whenever there is a need to differentiate between two labels. For example, given a label $L^f = (P^f,w^f,s^f)$ with superscript $f$, we will refer to its attributes as $M^f$, $S^f_v$, and so on. Given two walks $P$ and $P'$ with $P'$ starts at the ending vertex of $P$, define $P\oplus P'$ to be the walk obtained by concatenating $P'$ to $P$, that is, if $P=(i_1,\ldots, i_k)$ and $P'=(i_k,\ldots, i_l)$ then $P\oplus P' = (i_1,\ldots, i_l)$.

The initial label $L^0=(\PL^0,w^0,s^0)$ is created by setting $\PL^0=(0)$, $w^0=0$, and $s^0=0$. Table \ref{Label-table} illustrates the initial values of attributes associated with $L^0$. When a label $L=(\PL,w,s)$ with $\PL=(i_0,\ldots,i_h)$ is extended to a customer $j$, we set $j=i_{h+1}$ and create \emph{three new labels} depending on whether $j$ is active or seamless: $L^1 = (\PL\oplus (i_h, j), h+1, a_j)$ with customer $j$ being active and served at $a_j$, $L^2 = (\PL \oplus (i_h, j), h+1, b_j)$ with customer $j$ being active and served at $b_j$, and $L^3 = (\PL\oplus (i_h, j),w,s)$ with customer $j$ being seamless and the last known active customer remains to be $i_w$. Table \ref{Label-table} illustrates how the attributes of the three labels are updated according to the attributes of $L$. We discard the newly generated label, whenever the set $S_v$ of feasible speeds becomes empty or the optimization problem for computing the minimal cost $F_{\sd}$ during the label extension becomes infeasible.
\begin{table}[ht]
\begin{center}
\begin{tabular}{|c|c|c|c|c|}
\hline  &   & \multicolumn{3}{c|}{Label extension from $L$} \\
\cline{3-5}
        &  $L^0$ & $L^1$ & $L^2$ & $L^3$ \\ \hline
$M$  & $\emptyset$ & \multicolumn{3}{|c|}{See details below} \\
\hline
$\mu(P)$ & 0 & \multicolumn{3}{|c|}{$\mu(P)+\mu_j$} \\
\hline
$q(P)$ & 0 & \multicolumn{3}{|c|}{$q(P)+q_j$} \\
\hline
$S_v$ & $[l,u]$ & $[l,u]$ & $[l,u]$ & $S_v \cap \{v \mid a_j \le s+\tau_{i_w,i_h} + \frac{D_{i_w,j}}{v} \le b_j\}$\\
\hline
$\Gamma$ & 0  &\multicolumn{3}{|c|}{$\Gamma + \tau_j$}\\
\hline
$D$ & 0 &  \multicolumn{3}{|c|}{$D+ d_{i_h,j}$}\\
\hline
$F_{\sd}$ & 0 & $F_{\sd}+F^1$ & $F_{\sd}+F^2$ & $F_{\sd}$ \\
\hline
\end{tabular}
\caption{Label initialization and extension from a label $L$.}\label{Label-table}
\end{center}
\end{table}

The update of set $M$ in Table \ref{Label-table} depends on the choice of route $r$ in the set partitioning formulation: $M \gets M\cup \{j\}$ for elementary routes, $M \gets \{j\} $ for q-routes, and $M \gets \{i_{h}, j\}$ for 2-cycle-free q-routes. The entry $F^1$ (or $F^2$) in Table~\ref{Label-table} denotes the optimal cost when the vehicle leaves customer $i_w$ at time $s$, travels at some speed in $S_v$ to customer $j$ without any waiting in between, and serves customer $j$ at time $a_j$ (or $b_j$). The value of $F^1$ is calculated through solving a one-dimensional convex optimization problem, as illustrated below.
\begin{equation} \label{eq:hj3}
F^1 =\min \left\{D_{i_w, j}f(v) \mid v\in S_v, \; s + \tau_{i_w, i_h}  + \frac{D_{i_w, j}}{v} \le a_j\right\}.
\end{equation}
The feasible region of problem~\eqref{eq:hj3} is a closed interval $S_v\cap \left[\frac{D_{i_w, j}}{a_j - s - \tau_{i_w, i_h}}, +\infty \right)$, so the optimal solution of~\eqref{eq:hj3} must be attained at one of its endpoints, or $v_F$ if $v_F \in S_v$. The calculation of $F^2$ is similar to that of $F^1$:
	\begin{equation} \label{eq:hj4}
F^2=\min \left\{ D_{i_w,j}f(v) \mid  v\in S_v, \; s + \tau_{i_w, i_{h}}  + \frac{D_{i_w, j}}{v} = b_j\right\}. 
\end{equation}
Thus $F^2=D_{i_w,j}f\left(\frac{D_{i_w, j}}{b_j - s - \tau_{i_w, i_h}}\right)$ if $\frac{D_{i_w, j}}{b_j - s - \tau_{i_w, i_h}}\in S_v$, otherwise the label is discarded. We also discard a label whenever its attribute $q(P)$ exceeds the vehicle capacity $Q$.

\subsection{Label termination}
Label termination refers to the extension of a label $L$ to the depot, in which the vehicle leaves customer $i_w$ at time $s$ and travels at a uniform speed $v\in S_v$ along $P$ back to the depot. Let the minimum cost between $i_w$ and the depot along $\PL$ be $F_{\dpt}$. It can be computed through solving a one-dimensional convex optimization below.
\begin{equation} \label{eq:vfd}
F_{\dpt}=\min \left\{D_{i_w,0}f(v) \mid v\in S_v, \; s + \tau_{i_w,i_h} + \frac{D_{i_w,0}}{v} \le b_0\right\}.
\end{equation}
Label $L$ is discarded if~\eqref{eq:vfd} is infeasible. The total cost of a terminated label $L$ is
\begin{equation}
c_L := F_{\sd} + F_{\dpt}.
\end{equation}

\noindent The correctness of the labeling algorithm is shown by Proposition \ref{label-correctness} below.
\begin{proposition}\label{label-correctness}
Suppose a terminated label $L$ recursively defines a route $r$, a set $I$ of active indices over $r$, and the corresponding service start times $\vec{s}$ at the active customers. Then the cost $c_{r, I, \vec{s}}$ is given by $c_L$.
\end{proposition}
\begin{proof}
According to Proposition \ref{prop:sop:speed}, given a route and an optimal speed vector over the route, a customer is either active or seamless. The optimal speed at which the vehicle travels from active customer $i$ to active customer $j$, through a sequence of seamless customers, can be computed independent of the speed decisions before customer $i$ and after customer $j$. The way each label is extended and terminated guarantees the optimal speed over each arc is computed correctly, so we have $c_{r,I,\vec{s}} =c_L$. 
\end{proof}

\section{Dominance Rules} \label{sec:dominance}
The efficiency of any labeling algorithm heavily relies on dominance rules,
which allows one to discard a significant number of (potentially
exponentially many) labels. Without such dominance rules, the
implementation of most pricing algorithms would be impractical.
We devote this section to new dominance rules that can be applied
to the labels defined in Section~\ref{sec:labeling}. Intuitively, label 1 is dominated by label 2 if any possible extension of label 1 cannot lead to a triple $(r, I, \vec{s})$ with a smaller reduced cost than some extension of label 2. We formalize below what we mean by ``possible extension'' and ``lead to.''

We first introduce two new notations for label $L^f=(P^f, w^f, s^f)$. Let $i^f$ denote the last vertex of $P^f$. Let $\bar{v}^f$ denote the optimal speed vector over arcs from the depot to $i_{w^f}$ along $P^f$, obtained during the update of $F_{\sd}$. In other words, the component of $\bar{v}^f$ is computed by recursively solving the optimization problem~\eqref{eq:hj3} or~\eqref{eq:hj4} during the generation of label $L^f$.

\begin{definition}
Let $P$ be a walk and $v_P$ be a speed vector over arcs of $P$. Given a label $L^f$, the pair $(P, v_P)$ is a \emph{feasible extension} of $L^f$, if the following conditions hold:
\begin{itemize}
	\item The walk $P$ starts at vertex $i^f$, $P^f\oplus P$ is a route, and the total demands on $P^f\oplus P$ do not exceed the vehicle capacity.
	\item There exists $v\in S_v^f$ such that $(\bar{v}^f, v, v_P)$ is a feasible speed vector over the route $P^f \oplus P$. That is, a vehicle can travel from the depot to vertex $i_{w^f}$ at speed $\bar{v}^f$, travel from $i_{w^f}$ to $i^f$ at constant speed $v$, travel from $i^f$ to the depot along $P$ at speed $v_P$, and the time-window constraints of all the customers on $P^f \oplus P$ are satisfied.
\end{itemize}
\end{definition}
\noindent Note that if $(P, v_P)$ is a feasible extension of the label $L^f$, then the set of active customers on route $P^f \oplus P$ with the speed vector $(\bar{v}^f, v, v_P)$ for some $v\in S_v^f$ will be consistent with those in label $L^f$. Let $E(L^f)$ denote the set of all possible feasible extensions of $L^f$.
\begin{definition} \label{prop:domsimple}
A label $L^2=(P^2, w^2, S_v^2)$ is a \emph{dominated} label (so it can be discarded), if
there exists a label $L^1=(P^1,w^1,S_v^1)$ such that the following conditions hold:
\begin{enumerate}
  \item $E(L^2) \subseteq E(L^1)$.
  \item For any $(P, v_P) \in E(L^2)$ and $v\in S_v^2$ such that $(\bar{v}^2, v, v_P)$ is a feasible speed vector over $P^2 \oplus P$, let $I^2$ be the set of active customers on $P^2 \oplus P$ with this speed vector, and let $\vec{s}^2$ be the vector of the corresponding service start times at these customers. There exists a set $I^3$ of active customers on $P^1\oplus P$ with the corresponding service start times $\vec{s}^3$ such that $\bar{c}_{P^1 \oplus P, I^3, \vec{s}^3} < \bar{c}_{P^2\oplus P, I^2, \vec{s}^2}$.	
\end{enumerate}
\end{definition}
\noindent Definition~\ref{prop:domsimple} does not state that $L^2$ is dominated by $L^1$. Rather, it states that $L^2$ is dominated by some label $L^3$ with $P^3 = P^1$ through comparing $L^1$ and $L^2$. Indeed the conditions in Definition~\ref{prop:domsimple} cannot guarantee that extending $L^1$ with speed vector $\bar{v}^1$ always leads to a column with a better reduced cost than any extension of $L^2$. This is another feature that is different from the usual domination conditions for regular labels that are defined just by routes. The existence of such a label $L^1$ in Definition~\ref{prop:domsimple}, however, is in general hard to assert unless we have already generated such a label $L^1$ and can check $L^2$ against it. We will replace the conditions in Definition~\ref{prop:domsimple} by a set of
sufficient conditions that are easier to check.
 
We first introduce two functions needed for the sufficient conditions. Given a label $L=(P, w, s)$, let $T: S \rightarrow \Re$ be the function that computes the service finish time at the last vertex $i_h$ when the vehicle travels at speed $v$ between vertex $i_w$ and vertex $i_h$ along $\PL$. Specifically,
\begin{equation} \label{eq:T}
T(v)=s+\Gamma + \frac{D}{v}.
\end{equation}
Let $C: S \rightarrow \Re$ be the function that computes the total cost up to vertex $i_h$ along $P$ at a uniform speed $v$ between vertex $i_w$ and vertex $i_h$. Specifically,
\begin{equation} \label{eq:C}
C(v)= F_{\sd} + D \cdot f(v).
\end{equation}

\noindent We propose a new dominance rule as follows.
\begin{proposition} \label{prop:dom}
The label $L^2=(P^2,w^2,S_v^2)$ is a dominated label, if there exists a label $L^1=(P^1,w^1,S_v^1)$ such that the following conditions hold:
\begin{enumerate}
  \item $i^1=i^2$; \label{prop:dom:lastvx} \label{prop:dom:endpoint}
  \item $M^1 \subseteq M^2$; \label{prop:dom:allowed}
  \item $q(P^2) \geq q(P^1)$; \label{prop:dom:load}
  \item For any $v_2 \in S_v^2$, there exists $v_1 \in S_v^1$ such that \label{prop:dom:speed}
    \begin{equation}T^1(v_1) \le T^2(v_2)\tag{\ref{prop:dom:speed}-a}
    \label{eqn:dom:speed:time}
    \end{equation}
    and 
     \begin{equation}
		C^1(v_1)  - \mu( P^1 )  < C^2(v_2) - \mu(P^2).
      \tag{\ref{prop:dom:speed}-b}
      \label{eqn:dom:speed:cost}
      \end{equation}	
\end{enumerate}
\end{proposition}

\begin{proof}
See the appendix.
\end{proof}

We discuss in detail how to check the dominance rule in Proposition~\ref{prop:dom}. Given two labels $L^1$ and $L^2$ that end at the same vertex, conditions \ref{prop:dom:endpoint}--\ref{prop:dom:load} are easy to check. We divide checking if condition \ref{prop:dom:speed} holds into four cases. To simplify notation, let the last vertex in both labels be $j$, the index of vertex $j$ on $P^1$ be $h^1$, and the index of vertex $j$ on $P^2$ be $h^2$.
\subsection{Case 1: $w^1=h^1$ and $w^2 = h^2$}\label{case1}
In this case, vertex $j$ is the last active vertex on both $P^1$ and $P^2$. Therefore, $T^1(v)=s^1+\tau_j$, $T^2(v)=s^2+\tau_j$, $C^1(v)=F^1_{\sd}$, and $C^2(v)=F^2_{\sd}$. Condition \ref{prop:dom:speed} is satisfied if and only if $s^1\le s^2$ and $F^1_{\sd} < F^2_{\sd}$.

\subsection{Case 2: $w^1 < h^1$ and $w^2=h^2$}\label{case2}
The service start time $s^2$ at $j$ in label $L^2$ can only be $a_j$ or $b_j$, and $C^2(v)=F^2_{\sd}$. 
\begin{itemize}
	\item If $s^2=a_j$, considering condition \eqref{eqn:dom:speed:time} we must have
$T^1(v_1) = a_j+\tau_j$, which uniquely determines a value $v_1$. If $v_1\in S_v^1$, then we check if \eqref{eqn:dom:speed:cost} holds by comparing two numbers.
	\item If $s^2=b_j$, then any $v_1\in S_v^1$ satisfies that $T^1(v_1) \le T^2(v_2) = b_j + \tau_j$. We then check if there exists $v_1\in S_v^1$ such that \eqref{eqn:dom:speed:cost} is satisfied, which is equivalent to check if the following inequality holds:
\[\min_{v_1\in S_v^1}C^1(v_1)-\mu( P^1 )  < F^2_{\sd} - \mu( P^2 ).\]

\noindent The term $\min_{v_1\in S_v^1}C^1(v_1)$ can be computed by minimizing a one-dimensional convex function $C^1(v_1)$ over the closed interval $S_v^1$.
\end{itemize}

\subsection{Case 3: $w^1= h^1$ and $w^2 < h^2$}\label{case3}
The service start time $s^1$ at $j$ in label $L^1$ can only be $a_j$ or $b_j$, and $C^1(v)= F^1_{\sd}$.
\begin{itemize}
	\item If $s^1 = b_j$, then $T^1(v_1) = b_j + \tau_j \geq T^2(v_2)$ for any $v_2\in S_v^2$.
Condition \eqref{eqn:dom:speed:time} holds only if $T^1(v_1) =b_j + \tau_j = T^2(v_2)$ for all $v_2 \in S_v^2$. It implies that $S_v^2$ must contain only a single element. Thus it is easy to check if \eqref{eqn:dom:speed:cost} holds.
	\item If $s^1=a_j$, then $T^1(v_1) = a_j + \tau_j \le T^2(v_2)$ for any $v_2\in S_v^2$.
We then check if for all $v_2\in S_v^2$ condition \eqref{eqn:dom:speed:cost} is satisfied, which is equivalent to check if the following inequality holds:
\[F^1_{\sd} - \mu( P^1 ) < \min_{v_2\in S_v^2}C^2(v_2) - \mu( P^2 ).\]
The term $\min_{v_2\in S_v^2}C^2(v_2)$ can be computed by minimizing a one-dimensional convex function $C^2(v_2)$ over the closed interval $S_v^2$.
\end{itemize}

\subsection{Case 4: $w^1 < h^1$ and $w^2 < h^2$}\label{case4}
This is the most tricky case. We first assume $S_v^1=[v^{\min}_1, v^{\max}_1]$ and $S_v^2=[v^{\min}_2, v^{\max}_2]$. The following proposition shows checking if condition~\eqref{eqn:dom:speed:time} holds can be done by comparing two numbers.
\begin{proposition} \label{prop:condition4a:equiv}
Condition~\eqref{eqn:dom:speed:time} in Proposition~\ref{prop:dom} holds if and only if $T^1(v^{\max}_1) \le T^2(v^{\max}_2)$.
\end{proposition}
\begin{proof}
Note that $T^1$ and $T^2$ are both monotonically decreasing. Suppose for any $v_2\in [v^{\min}_2, v^{\max}_2]$, there exists $v_1 \in [v^{\min}_1, v^{\max}_1]$ such that $T^1(v_1) \le T^2(v_2)$. Then for $v^{\max}_2$, there exists $v_1 \in [v^{\min}_1, v^{\max}_1]$ such that $T^1(v_1) \le T^2(v^{\max}_2)$. Thus $T^1(v^{\max}_1) \le T^1(v_1) \le T^2(v^{\max}_2)$. On the other hand, if $T^1(v^{\max}_1) \le T^2(v^{\max}_2)$, then for any $v_2 \in [v^{\min}_2, v^{\max}_2]$, $v^{\max}_1 \in [v^{\min}_1, v^{\max}_1]$ satisfies that $T^1(v^{\max}_1) \le T^2(v^{\max}_2) \le T^2(v_2)$.
\end{proof}

To check condition~\eqref{eqn:dom:speed:cost} in Proposition~\ref{prop:dom}, we need the following definitions.
\begin{definition} $\;$
\begin{itemize}
	\item Define  
	\[\beta(v) = \frac{D^1}{T^2(v) - s^1 - \Gamma^1}= \frac{D^1}{D^2/v + s^2 + \Gamma^2 -  s^1 - \Gamma^1} = \frac{D^1}{D^2/v + \delta},\]
where $\delta =  s^2 + \Gamma^2 - s^1 - \Gamma^1$ is a constant.
	\item When $\delta \neq 0$, define $v^* = (D^1 - D^2)/\delta$. We say $v^*$ is well defined if $\delta \neq 0$.
	\item Define $H(v)= C^1(\beta(v)) - C^2(v)$.
\end{itemize}
\end{definition}
\begin{remark}
The function $\beta$ computes the speed at which the vehicle travels from $i_{w^1}$ to $j$ along $P^1$ and arrives at customer $j$ at the same time as the vehicle travels from $i_{w^2}$ along $P^2$ at speed $v$. The speed $v^*$ is the only value such that $\beta(v)=v$ when $\delta\neq 0$. The function $H$ captures the cost difference between two labels when the vehicle travels at speed $v$ through $P^2$ and travels at speed $\beta(v)$ through $P^1$ (thus arriving at customer $j$ at the same time).   
\end{remark}

\noindent Then condition~\eqref{eqn:dom:speed:cost} in Proposition~\ref{prop:dom} can be checked by simply comparing numbers based on the result below.
\begin{proposition}
Suppose that condition~\eqref{eqn:dom:speed:time} holds, i.e., $T^1(v^{\max}_1) \le T^2(v^{\max}_2)$. Condition~\eqref{eqn:dom:speed:cost} in Proposition~\ref{prop:dom} holds if and only if $z^* - \mu( P^1 ) + \mu(P^2) < 0$, where the value of $z^*$ is calculated as follows.
\begin{enumerate}
	\item If $T^1(v^{\min}_1) \le T^2(v^{\max}_2)$, then $z^*=C^1(v^{\min}_1) - C^2(v^{\min}_2)$.
	\item If $T^2(v^{\max}_2) < T^1(v^{\min}_1) \le T^2(v^{\min}_2)$,	
	\begin{itemize}
		\item and if $v^*$ is well-defined and $v^* \in [D^2 v^{\min}_1/(D^1 - \delta v^{\min}_1), v^{\max}_2]$, then $z^*=\max\{C^1(v^{\min}_1) -  C^2(v^{\min}_2), H(v^{\max}_2), H(v^*)\}$.
		\item and if $v^*$ is not well-defined or $v^* \notin [D^2 v^{\min}_1/(D^1 - \delta v^{\min}_1), v^{\max}_2]$, then $z^*=\max\{C^1(v^{\min}_1) -  C^2(v^{\min}_2), H(v^{\max}_2)\}$.
	\end{itemize}		
	\item If $T^2(v^{\min}_2) < T^1(v^{\min}_1)$,	
	\begin{itemize}
		\item and if $v^*$ is well-defined and $v^*\in [v^{\min}_2, v^{\max}_2]$, then $z^*=\max\{H(v^{\min}_2), H(v^{\max}_2), H(v^*)\}$.
		\item and if $v^*$ is not well-defined or $v^*\notin [v^{\min}_2, v^{\max}_2]$, then $z^*=\max\{H(v^{\min}_2), H(v^{\max}_2)\}$. 
	\end{itemize}	
\end{enumerate}
\end{proposition}
\begin{proof}
To check if Condition \eqref{eqn:dom:speed:cost} holds is equivalent to solve a one-dimensional constrained optimization problem with decision variable $v_2$. To see this, fix $v_2\in [v^{\min}_2, v^{\max}_2]$ and define
\begin{subequations}
\begin{align}
\phi(v_2) = \min \ & C^1(v_1) \\
\text{s.t.} \ & T^1(v_1)= s^1 + \Gamma^1 + \frac{D^1}{v_1} \le s^2 + \Gamma^2 + \frac{D^2}{v_2}=T^2(v_2)  \label{eq:phi2}\\
& v^{\min}_1 \le v_1 \le v^{\max}_1.
\end{align}
\end{subequations}
Constraint~\eqref{eq:phi2} is equivalent to $v_1\ge \beta(v_2)$. Based on the assumption that Condition \eqref{eqn:dom:speed:time} holds and Proposition \ref{observe} in the Appendix, we have $\beta(v^{\max}_2) \leq v^{\max}_1$. Thus $\beta(v_2) \leq v^{\max}_1$ for any $v_2\in S_v^2$. Therefore, $\phi(v_2)=C^1(v^{\min}_1)$ if $\beta(v_2) \le v^{\min}_1$ and $\phi(v_2) = C^1(\beta(v_2))$ otherwise. Define
\begin{subequations}
\begin{align*}
\psi(v_2)= \phi(v_2) - C^2(v_2)\\
z^* = \max\{\psi(v_2) \mid v_2 \in [v^{\min}_2, v^{\max}_2]\}
\end{align*}
\end{subequations}
Then~\eqref{eqn:dom:speed:cost} is satisfied if and only if
\[z^* - \mu(P^1) + \mu( P^2 ) < 0.\]
We consider the following three cases to compute $z^*$.
\begin{enumerate}
	\item $T^1(v^{\min}_1) \le T^2(v^{\max}_2)$. Then for any $v\in [v^{\min}_2, v^{\max}_2]$, $v^{\min}_1 \ge \beta(v)$ and $\phi(v)=C^1(v^{\min}_1)$. Thus $\psi(v)= C^1(v^{\min}_1) - C^2(v)$ and $z^*= C^1(v^{\min}_1) - C^2(v^{\min}_2) $.
	
	\item $T^2(v^{\max}_2) < T^1(v^{\min}_1) \le T^2(v^{\min}_2)$. Then $\beta(v^{\min}_2) \le  v^{\min}_1 < \beta(v^{\max}_2) $. Since $\beta$ is monotonically increasing by Proposition \ref{observe} in the Appendix, there exists $\tilde{v} \in [v^{\min}_2, v^{\max}_2]$ such that $\beta(\tilde{v}) = v^{\min}_1$. In particular, $\tilde{v}=D^2 v^{\min}_1/(D^1 - \delta v^{\min}_1)$. Then $\phi(v) = C^1(v^{\min}_1)$ for $v \in [v^{\min}_2, \tilde{v}]$, and $\phi(v) = C^1(\beta(v))$ for $v \in [\tilde{v}, v^{\max}_2]$.
	
	\noindent When $v \in [v^{\min}_2, \tilde{v}]$, the function $\psi(v)=C^1(v^{\min}_1) - C^2(v)$. Its maximum is attained at $C^1(v^{\min}_1) -  C^2(v^{\min}_2)$. When $v \in [\tilde{v}, v^{\max}_2]$, the function $\psi(v)=C^1(\beta(v)) - C^2(v)= H(v)$, which attains its maximum at either $\tilde{v}, v^{\max}_2$, or $v^*$ if it is well defined and $v^* \in [\tilde{v}, v^{\max}_2]$. Note that $\psi(\tilde{v}) = C^1(v^{\min}_1) - C^2(\tilde{v}) \le C^1(v^{\min}_1) -  C^2(v^{\min}_2)$. Therefore from Lemmas~\ref{lem:H1} and~\ref{lem:H2} in the Appendix, $z^*$ equals to the maximum of $C^1(v^{\min}_1) -  C^2(v^{\min}_2)$, $C^1(\beta(v^{\max}_2)) - C^2(v^{\max}_2)$, or $C^1(v^*) - C^2(v^*)$ if $v^*$ is well-defined and $v^* \in [\tilde{v}, v^{\max}_2]$.	
	
	\item $T^2(v^{\min}_2) < T^1(v^{\min}_1)$. Then for any $v_2 \in [v^{\min}_2, v^{\max}_2]$, $\beta(v) > v^{\min}_1$ and $\phi(v)=C^1(\beta(v))$. Thus $\psi(v) = C^1(\beta(v)) - C^2(v)= H(v)$ for any $v\in [v^{\min}_2, v^{\max}_2]$. Therefore from Lemmas~\ref{lem:H1} and~\ref{lem:H2} in the Appendix, the maximum of $H(v)$ is attained at $v^{\min}_2$, $v^{\max}_2$, or $v^*$ if $v^*$ is well-defined and $v^* \in [v^{\min}_2, v^{\max}_2]$. 
\end{enumerate} 
\end{proof}

\section{Other Components of The BCP Algorithm} \label{sec:other}
\paragraph{The choice of route $r$} We consider three variants of the set-partitioning formulation introduced in Section~\ref{sec:problem}. These variants differ on what $r$ represents in each triple $(r, I, \vec{s})$: an elementary route, a q-route, or a 2-cycle-free q-route. The linear programming relaxation with elementary routes provides the tightest lower bound, but usually the corresponding pricing problem takes the longest time to solve; on the other hand, the linear programming relaxation with q-routes gives the weakest lower bound, but the pricing problem takes the least time to solve. 

\paragraph{Cutting planes} At each node of the branch-and-bound tree, we add the rounded capacity inequalities for standard capacitated VRP in the following form
\[ \sum_{i \in S}{\sum_{j \notin S}{x_{ij}}} \geq \xi(S),\] 
where $S$ is a subset of customers and $\xi(S)$ is a lower bound on the minimum number of vehicles needed to visit customers in set $S$~\citep{toth2014vehicle}. These inequalities are separated via a heuristic from package \citep{lysgaard2003cvrpsep}. In our implementation, these inequalities are added in terms of the route variables $z_{r,I,\vec{s}}$, using the relationship $x_{ij} = \sum_{r\in \Omega_2}\beta_{ijr}z_{r,I,\vec{s}}$, where $\beta_{ijr}$ indicates the number of times route $r$ traverses arc $(i,j)$.

\paragraph{Branching rules} We branch on the arc variable $x_{ij}$, which is implicitly enfored with inequalities $\sum_{r\in \Omega_2}\beta_{ijr}z_{r,I,\vec{s}} \le 0$ and $\sum_{r\in \Omega_2}\beta_{ijr}z_{r,I,\vec{s}} \ge 1$. The structure of the pricing problem will remain the same after branching. 


\section{Computational Experiments} \label{sec:computation}
We conduct extensive computational experiments of the proposed BCP algorithm, and compare its performance with that of a mixed integer convex programming formulation \citep{Fuka2015TS} solved by a state-of-the-art optimization solver CPLEX. We test on two set of instances in the context of maritime transportation and road transportation, respectively. All the algorithms are coded in {\tt C++}. Computational experiments are conducted on an AMD machine running at $2.3$ GHz with $264$ GBytes of RAM memory, under Linux Operating System. We use SCIP Optimization Suite 3.2.0 for the branch and price framework. The time limit for each instance is set to 3600s. We first present in Section~\ref{sec:maritime} and Section~\ref{sec:road} the results of the best variant of our BCP algorithms: elementary routes for maritime transportation instances and 2-cycle-free routes for road transportation instances. Then we show in Section~\ref{sec:route_choice} the performances of different variants of our BCP algorithm. All reported computational time is in seconds.

\subsection{Maritime transportation} \label{sec:maritime}
\subsubsection{Instance generation} \label{testMaritime}
We generated test instances based on the instance generator for industrial 
and tramp ship routing and 
scheduling problems (TSRSP) available in~\cite{benchmark}. 
However, we need to modify such instances since that 
problem has several differences to the JRSP, for instance, the same port can
appear in multiple origin/destinations of orders, whereas in the JRSP, a given
client has a single demand. 

Thus we proceed as follows to generate a JRSP instance. The set of clients of
JRSP is
the set of ports in the TSRSP. 
We set the demand of a client (port) to be the first demand of an order containing that port as an origin and pick the corresponding time window and service time of that port according to that order.  


We use the bunker consumption rate function $f(v) = 0.0036v^2 - 0.1015v + 0.8848$ (per nautical mile) as the cost function over each arc. This function is based on real bunker consumption data for a specific ship~\citep{fagerholt2010reducing}.
We also set the speeds lower and upper bounds at 14 and 20 knots
respectively, since that is the range for which the fuel consumption function is
valid~\citep{fagerholt2010reducing}.


The instances in \cite{benchmark} were divided into two families: deep sea shipping and
short sea shipping. In deep sea shipping, the cargoes are transported long distances and
across at least one of the big oceans, giving long sailing time. In short sea shipping,
the operation is only regional. We generate short sea instances with 30 and 39 customers, 
and deep sea instances with 30, 40, and 50 customers. 
For each instance type with a given number of customers, we generate 5 random instances. 
Therefore we has 25 test instances in total for the maritime transportation problems.
The instances can be downloaded at \url{http://www.menet.umn.edu/~qhe/}.
Each instance is referred to in the format: \texttt{type}-$n$-$i$, where \texttt{type} is
either ``deep'' or ``short'' meaning either the deep sea or short sea instances, $n$ is
the number of customers and $i$ means it is the $i$-th instance (out of five) with the same 
type and number of customers.

\subsubsection{Computational results}
We compare the computational performance of our algorithm with that of a mixed integer convex programming formulation \citep{Fuka2015TS} solved by a state-of-the-art optimization solver CPLEX. The results are presented in Table \ref{table-maritime}. Columns 2 to 4 show the best lower bound (BLB) and the best upper bound (BUB) evaluated by CPLEX and the associated computational time. The entry ``\texttt{tl}'' denotes that the instance is not solved to optimality within the time limit. Similar results are given in columns 5 to 7 for our algorithm.  Symbol ``-'' in the table means that our algorithm is not able to obtain the associated bound within the time limit.

\begin{table}
\small
\def\arraystretch{0.8}
\centering
\begin{tabular}{llllrlllr}
& & \multicolumn{3}{c}{CPLEX} & & \multicolumn{3}{c}{BCP} \\
\cline{3-5} \cline{7-9}
Instance & & BLB & BUB & Time & & BLB & BUB & Time \\
\hline
deep-30-1 & & 14356.2 & 14356.2 & 43.5 & & 14356.3 & 14356.3 & 2.1\\
deep-30-2 & & 17943.3 & 17943.3 & 80.2 & & 17943.4 & 17943.4 & 94.3 \\
deep-30-3 & & 10812.3 & 10812.3 & 14.0 & & 10812.3 & 10812.3 & 0.2 \\
deep-30-4 & & 15166.7 & 15166.7 & 94.7 & & 15166.7 & 15166.7 & 0.6 \\
deep-30-5 & & 16268.0 & 16268.0 & 77.1 & & 16268.1 & 16268.1 & 1.0 \\
deep-40-1 & & 15738.3 & 15738.3 & 395.4 & & 15738.3 & 15738.3 & 5.34 \\
deep-40-2 & & 14565.6 & 14859.2 & \texttt{tl} & & 14859.3 & 14859.3 & 59.6 \\
deep-40-3 & & 18877.3 & 18877.3 & 1047 & & 18877.4 & 18877.4 & 200 \\
deep-40-4 & & 18408.3 & 18408.3 & 516.1 & & 18408.3 & 18408.3 & 365 \\
deep-40-5 & & 18065.3 & 18065.3 & 1746 & & 18065.4 & 18065.4 & 3.4  \\
deep-50-1 & & 13957.9 & 19930.1 & \texttt{tl} & & - & - & \texttt{tl} \\
deep-50-2 & & 11417.2 & 19323.2 & \texttt{tl} & & - & 17918.0 & \texttt{tl} \\
deep-50-3 & & 15633.3 & 19650.6 & \texttt{tl} & & 19649.4 & 134609.5 & \texttt{tl} \\
deep-50-4 & & 16345.8 & 19589.6 & \texttt{tl} & & 19538.9 & 19538.9 & 2095 \\
deep-50-5 & & 10622.1 & 17811.3 & \texttt{tl} & & - & 17605.6 & \texttt{tl} \\
\hline
short-30-1 & & 3091.7 & 3091.7 & 16.0 & & 3091.7 & 3091.7 & 0.8 \\
short-30-2 & & 2814.8 & 2814.8 & 33.9 & & 2814.8 & 2814.8 & 2.0 \\
short-30-3 & & 3508.4 & 3508.4 & 231.2 & & 3508.4 & 3508.4 & 336.1 \\
short-30-4 & & 3137.5 & 3137.5 & 12.9 & & 3137.5 & 3137.5 & 5.42 \\
short-30-5 & & 3081.6 & 3081.6 & 13.6 & & 3081.6 & 3081.6 & 1.1 \\
short-39-1 & & 3643.0 & 3643.0 & 744.3 & & 3643.0 & 3643.0 & 41.8 \\
short-39-2 & & 3271.7 & 3569.4 & \texttt{tl} & & 3451.2 & 3772.6 & \texttt{tl} \\
short-39-3 & & 3627.8 & 3627.8 & 1315 & & 3594.7 & 4024.1 & \texttt{tl} \\
short-39-4 & & 3663.3 & 3663.3 & 567.4 & & 3663.3 & 3663.3 & 594.6 \\
short-39-5 & & 3437.2 & 3437.2 & 108.0 & & 3437.2 & 3437.2 & 79.8 \\
\hline
\end{tabular}
\caption{Comparing BCP algorithm and CPLEX on solving Maritime instances} \label{table-maritime}
\end{table}

We see from Table \ref{table-maritime} that our algorithm is able to provide optimality certificates for 19 out of 25 instances, whereas CPLEX can only do so for 18 instances. Among the 17 instances that are solved to optimality by both algorithms within the time limit, the average solution times for CPLEX and our algorithm are 337s and 102s, respectively. Comparing the performances in more detail separately for deep see and short sea instances, it is clear that our algorithm significantly outperforms CPLEX for the deep sea instances, while CPLEX gets more competitive for the short sea instances. In particular, our algorithm solves 11 out of 15 deep sea instances to optimality, with an average solution time of 1148s (the solution time of each unsolved instance is counted as 3600s), while CPLEX solves 9 out of 15 deep sea instances to optimality, with an average solution time of 1708s. For short sea instance, our algorithm takes less solution time for 6 out 10 instances, while CPLEX is faster in 3 out of 10 instances.  

\subsection{Road transportation}  \label{sec:road}
We next test our algorithm on the PRP proposed in~\cite{bektacs2011pollution}. The speed-dependent fuel cost function $f(v)=\frac{\pi_1}{v}+\pi_2 v^2$ with $\pi_1= 1.42\times 10^{-3}$ (\textit{\pounds/s}) and $\pi_2=1.98 \times 10^{-7}$ (\textit{\pounds s$^2$/m$^3$}). The PRP also has a load-dependent cost component and a cost component linear in the total travel time over the route; see~\cite{bektacs2011pollution} for detailed cost calculation of a route in the PRP. The PRP can be formulated as a mixed-integer second-order cone program, which can be handled directly by CPLEX~\citep{Fuka2015TS}. Our BCP algorithm needs to include an additional load-dependent component in Condition \eqref{eqn:dom:speed:cost} when checking if a label is dominated.

\subsubsection{Test instances and preprocessing}
We test our algorithm on the benchmark instances in~\cite{bektacs2011pollution} and \cite{kramer2014matheuristic}, generated using the geographical locations of United Kingdom cities. The instances can be divided into three groups, UK-A, UK-B, and UK-C. UK-A instances have the widest time windows, UK-B instances have the narrowest time windows, and UK-C instances have time windows of moderate length. We use UK$n$\texttt{G}-$i$ to reference each instance, where $n$ is the number of customers, $\texttt{G} \in \{A,B,C\}$, and $i$ is the $i$-th instance with the same type and number of customers. For example, UK20A-1 represents the first instance in the UK-A group with 20 customers.

Before solving the instances with our algorithm and CPLEX, we tighten the time windows of each UK instance following a similar procedure for the TSPTW, described in~\cite{ascheuer2001solving} and \cite{dash2012time}. For each instance, we set the speed lower limit $l_{ij}$ to be $(\frac{\pi_1}{2\pi_2})^{\frac{1}{3}}$, at which the speed-dependent fuel cost function attains its minimum. We repeat the following tightening steps until no more changes can be made:
\begin{itemize}	
	\item $a_k \leftarrow \max\{a_k, \min_{i\in \delta^-(k)}\{a_i+\tau_i+\frac{d_{ik}}{u_{ik}}\}\}$ for $k\in V$,
	\item $a_k \leftarrow \max\{a_k, \min_{j \in \delta^+(k)}\{a_j-\frac{d_{kj}}{l_{kj}}-\tau_k\}\}$ for $k\in V$,
	\item $b_k \leftarrow \min \{b_k, \max\{a_k, \max_{i\in \delta^-(k)}\{b_i +\tau_i+\frac{d_{ik}}{l_{ik}}\}\}\}$ for $k\in V$,
	\item $b_k \leftarrow \min \{b_k, \max_{j\in \delta^+(k)}\{b_j - \frac{d_{kj}}{u_{kj}}-\tau_k\}\}$ for $k\in V$.
\end{itemize}

\subsubsection{Computation results}
We present in Table \ref{table-PRPlib} the results for the PRP instances with 10 and 20 customers by our BCP algorithm and the branch-and-cut algorithm in~\cite{Fuka2015TS}, labeled as BC. 

\begin{table}[ht]
\small
\def\arraystretch{0.8}
\centering
\begin{tabular}{lllrrllllrllr}
& & & & & & & & \multicolumn{2}{c}{BC} & & \multicolumn{2}{c}{BCP} \\
\cline{9-10} \cline{12-13}
Instance & & Opt. & \begin{tabular}{l} BC\\Time \end{tabular} & \begin{tabular}{l} BCP\\Time \end{tabular} & & Instance & & Best & \begin{tabular}{l} Time/\\Gap \end{tabular} & & Best & \begin{tabular}{l} Time/\\Gap \end{tabular} \\
\cline{1-5} \cline{7-13}
UK10A-1 & & 170.64 & 1354 & 0.4 & & UK20A-1 & & 352.45 & 22.9\% & & 351.82 & 13.9 \\
UK10A-2 & & 204.88 & 813 & 0.1 & & UK20A-2 & & 365.77 & 20.7\% & & 365.77 & 1.8 \\
UK10A-3 & & 202.56 & 1708 & 0.2 & & UK20A-3 & & 230.49 & 23.6\% & & 230.49 & 11.8 \\
UK10A-4 & & 189.88 & 844 & 0.2 & & UK20A-4 & & 347.04 & 21.2\% & & 347.04 & 43.5 \\
UK10A-5 & & 175.59 & 2649 & 0.2 & & UK20A-5 & & 329.63 & 24.3\% & & 323.44 & 16.6 \\
UK10A-6 & & 214.48 & 1472 & 0.04 & & UK20A-6 & & 367.73 & 25.0\% & & 364.23 & 18.1 \\
UK10A-7 & & 190.14 & 882 & 0.07 & & UK20A-7 & & 258.75 & 23.3\% & & 253.61 & 0.06\% \\
UK10A-8 & & 222.17 & 564 & 0.02 & & UK20A-8 & & 303.17 & 23.0\% & & 301.51 & 12.7 \\
UK10A-9 & & 174.54 & 352 & 0.07 & & UK20A-9 & & 362.56 & 19.5\% & & 362.56 & 18.6 \\
UK10A-10 & & 189.82 & 211 & 0.1 & & UK20A-10 & & 317.79 & 26.3\% & & 313.33 & 13.7 \\
\cline{1-5} \cline{7-13}
UK10B-1 & & 246.44 & 0.2 & 0.01 & & UK20B-1 & & 469.35 & 1.1 & & 469.35 & 0.02 \\
UK10B-2 & & 303.73 & 0.2 & 0.01 & & UK20B-2 & & 477.05 & 4.8 & & 477.05 & 0.02 \\
UK10B-3 & & 301.89 & 0.1 & 0.01 & & UK20B-3 & & 354.46 & 0.9 & & 354.46 & 0.03 \\
UK10B-4 & & 273.90 & 0.3 & 0.01 & & UK20B-4 & & 523.59 & 2.8 & & 523.59 & 0.1 \\
UK10B-5 & & 255.07 & 0.2 & 0.01 & & UK20B-5 & & 447.33 & 2.3 & & 447.33 & 1.2 \\
UK10B-6 & & 332.34 & 0.2 & 0.01 & & UK20B-6 & & 511.78 & 6.7 & & 511.78 & 7.3 \\
UK10B-7 & & 314.64 & 0.6 & 0.03 & & UK20B-7 & & 379.02 & 3.0 & & 379.02 & 0.02 \\
UK10B-8 & & 339.36 & 0.2 & 0.01 & & UK20B-8 & & 431.31 & 1.6 & & 431.31 & 4.8 \\
UK10B-9 & & 261.10 & 0.2 & 0.01 & & UK20B-9 & & 548.68 & 1.7 & & 548.68 & 0.1 \\
UK10B-10 & & 285.20 & 0.2 & 0.01 & & UK20B-10 & & 410.32 & 1.6 & & 410.32 & 0.1 \\
\cline{1-5} \cline{7-13}
UK10C-1 & & 210.18 & 1.6 & 0.07 & & UK20C-1 & & 432.82 & 6.8\% & & 432.82 & 90.3 \\
UK10C-2 & & 271.93 & 72.4 & 0.01 & & UK20C-2 & & 450.35 & 12.7\% & & 448.29 & 4.4 \\
UK10C-3 & & 229.18 & 2.5 & 0.01 & & UK20C-3 & & 287.04 & 9.9\% & & 287.04 & 4.1 \\
UK10C-4 & & 230.52 & 5.2 & 0.01 & & UK20C-4 & & 434.23 & 7.3\% & & 434.23 & 4.5 \\
UK10C-5 & & 205.49 & 7.6 & 0.04 & & UK20C-5 & & 382.46 & 10.9\% & & 381.70 & 5.9 \\
UK10C-6 & & 255.82 & 17.5 & 0.01 & & UK20C-6 & & 444.35 & 14.9\% & & 444.35 & 28.1 \\
UK10C-7 & & 217.79 & 14.3 & 0.02 & & UK20C-7 & & 321.67 & 17.3\% & & 317.73 & 63.8 \\
UK10C-8 & & 251.29 & 4.7 & 0.01 & & UK20C-8 & & 410.35 & 1224 & & 410.35 & 0.5 \\
UK10C-9 & & 186.04 & 5.8 & 0.01 & & UK20C-9 & & 421.39 & 3347 & & 421.39 & 0.9 \\
UK10C-10 & & 231.62 & 3.8 & 0.02 & & UK20C-10 & & 390.68 & 7.2\% & & 384.88 & 13.7 \\
\cline{1-5} \cline{7-13}
\end{tabular}
\caption{Comparison between the BCP algorithm and the BC algorithm in~\cite{Fuka2015TS} on UK instances.}\label{table-PRPlib}
\end{table}

The first four columns show the results of both algorithms for instances with 10 customers. Both algorithms can solve these instances to optimality within the time limit. The column ``Opt.'' gives the optimal objective value for each instance. The next five columns show the performance of both algorithms on instances with 20 customers. Columns with label ``Best'' show the best objective value (upper bound) found by the two algorithms, and columns ``GAP/time'' give the total time (seconds) it takes to solve the instance to optimality if the time limit is not reached, and give the ending optimality gap otherwise. 

We can see that the proposed BCP algorithm significantly outperforms the previous BC algorithm for all instances considered. The BCP algorithm is capable of solving 59 out of 60 instances within the time limit, and takes 6.48s on average to solve these 59 instances. On the other hand, the BC algorithm can only solve 42 instances, takes 371.1 s on average to solve them, and the average optimality gap for the remaining 18 instances is as large as $17.6\%$. Note that there are four UK20-A instances and six UK20-C instances on which the BC algorithm finds the optimal solution but is not able to prove its optimality within the time limit.
 
We also tested our algorithm on instances with more customers. The results are summarized in Table~\ref{table-numOpt}, where an entry $n_1 (n_2)$ denotes that $n_1$ out of $n_2$ instances in the group are solved to optimality. Our BCP algorithm is able to solve 2 out of 10 UK-A instances and 5 out of 10 UK-C instances with 25 customers, 1 out of 10 UK-C instance with 50 customers, and 6 out of 10 UK-B instances with 50 customers. Note that from \ref{table-PRPlib}, the BC algorithm is not able to solve within the time limit any UK-A of 20 customers or most UK-C instances with 20 customers. We also observe in our experiment that the BC algorithm is not able to solve any UK-B instance with 50 customers. These results show that our BCP algorithm is much better at tackling larger-size instances. 

\begin{table}[htb]
\footnotesize
\def\arraystretch{0.6}
\centering
\begin{tabular}{ccccccc}
\multicolumn{2}{c}{Number of customers} & & 10 & 20 & 25 & 50 \\
\hline
& UK-A & & 10 (10) & 9 (10) & 2 (10) & 0 (10) \\
& UK-B & & 10 (10) & 10 (10) & 10 (10) & 6 (10) \\
& UK-C & & 10 (10) & 10 (10) & 5 (10) & 1 (10) \\
\hline
\end{tabular}
\caption{Number of PRP instances solved to optimality by our BCP algorithm.}\label{table-numOpt}
\end{table}

\subsection{The impact of route choices in our BCP algorithm} \label{sec:route_choice}
We investigate how different route choices in our set partitioning formulation affects the overall performance of the algorithm. The route choice in the set partitioning formulation affects the efficiency of the pricing algorithm as well as the strength of the linear programming relaxation bound of the corresponding formulation. We mainly study three different route choices: q-routes, 2-cycle-free q-routes, and elementary routes. 
The performance of the BCP algorithm with different route choices are shown in Table \ref{maritime_routechoice} and \ref{prp_routechoice}, for the maritime instances and the PRP instances, respectively.

For maritime transportation instances, the BCP algorithm with elementary routes works the best, and is able to solve 19 out of 25 instances. The performance of the BCP algorithm with 2-cycle-free q-routes is slightly worse, solving 12 out of 25 instances. We note that the BCP algorithm with q-routes cannot solve even the linear programming relaxation at the root node for most instances within time limit, and therefore the results are not shown in Table \ref{maritime_routechoice}.

\begin{table}[ht]
\small
\def\arraystretch{0.8}
\centering
\begin{tabular}{llllrlllr}
 & & \multicolumn{3}{c}{2-cycle-free} & & \multicolumn{3}{c}{Elementary} \\
\cline{3-5} \cline{7-9}
Instance & & BLB & BUB & time & & BLB & BUB & time \\
\hline
deep-30-1 & & 14356.3 & 14356.3 & 266.0 & & 14356.3 & 14356.3 & 2.1\\
deep-30-2 & & - & 22174.9 & \texttt{tl} & & 17943.4 & 17943.4 & 94.3 \\
deep-30-3 & & 10812.3 & 10812.3 & 0.9 & & 10812.3 & 10812.3 & 0.2 \\
deep-30-4 & & 15166.7 & 15166.7 & 8.6 & & 15166.7 & 15166.7 & 0.6 \\
deep-30-5 & & 16086.2 & 16268.1 & \texttt{tl} & & 16268.1 & 16268.1 & 1.0 \\
deep-40-1 & & 15738.3 & 15738.3 & 135.7 & & 15738.3 & 15738.3 & 5.34 \\
deep-40-2 & & 14668.3 & 20113.0 & \texttt{tl} & & 14859.3 & 14859.3 & 59.6 \\
deep-40-3 & & 18877.4 & 18877.4 & 71.3 & & 18877.4 & 18877.4 & 200 \\
deep-40-4 & & 18322.0 & 18408.3 & \texttt{tl} & & 18408.3 & 18408.3 & 365 \\
deep-40-5 & & - & 18065.4 & \texttt{tl} & & 18065.4 & 18065.4 & 3.4  \\
deep-50-1 & & - & - & \texttt{tl} & & - & - & \texttt{tl} \\
deep-50-2 & & - & - & \texttt{tl} & & - & 17918.0 & \texttt{tl} \\
deep-50-3 & & - & - & \texttt{tl} & & 19649.4 & 134609.5 & \texttt{tl} \\
deep-50-4 & & - & - & \texttt{tl} & & 19538.9 & 19538.9 & 2095 \\
deep-50-5 & & - & - & \texttt{tl} & & - & 17605.6 & \texttt{tl} \\
\hline
short-30-1 & & 3091.7 & 3091.7 & 19.0 & & 3091.7 & 3091.7 & 0.8 \\
short-30-2 & & 2814.8 & 2814.8 & 2.7 & & 2814.8 & 2814.8 & 2.0 \\
short-30-3 & & 3366.7 & 3747.2 & \texttt{tl} & & 3508.4 & 3508.4 & 336.1 \\
short-30-4 & & 3137.5 & 3137.5 & 46.4 & & 3137.5 & 3137.5 & 5.42 \\
short-30-5 & & 3081.6 & 3081.6 & 2.1 & & 3081.6 & 3081.6 & 1.1 \\
short-39-1 & & 3643.0 & 3643.0 & 118.5 & & 3643.0 & 3643.0 & 41.8 \\
short-39-2 & & 3432.2 & 3965.7 & \texttt{tl} & & 3451.2 & 3772.6 & \texttt{tl} \\
short-39-3 & & 3496.3 & 4501.7 & \texttt{tl} & & 3594.7 & 4024.1 & \texttt{tl} \\
short-39-4 & & 3663.3 & 3663.3 & 2077 & & 3663.3 & 3663.3 & 594.6 \\
short-39-5 & & 3437.2 & 3437.2 & 16.9 & & 3437.2 & 3437.2 & 79.8 \\
\hline
\end{tabular}
\caption{Comparison between the BCP algorithms with different route choices for the maritime instances}
\label{maritime_routechoice}
\end{table}

For the PRP instances, we test the performance of the BCP algorithm with three route choices on a total of 120 UK instances. The overall performance of the BCP algorithms with three different routes are given in Table \ref{prp_routechoice}. Note that when calculating the average computational time, we use the time limit 3600s for the instances not solved to optimality within the time limit. When calculating the average optimality gap, we use $100\%$ for the gap if the linear programming relaxation at the root node is not even solved within the time limit. We see from Table \ref{prp_routechoice} that the BCP algorithm with 2-cycle-free routes has the best performance overall. On the other hand, the BCP algorithm with elementary routes has the worst performance. The detailed results for all the instances are provided in Tables \ref{table-prp10} to \ref{table-prp50} in the Appendix.

\begin{table}[ht]
\small
\def\arraystretch{0.8}
\centering
\begin{tabular}{llllllll}
\multicolumn{2}{c}{} & & Elementary & & 2-cycle-free & & q-route \\
\hline
\multirow{2}{*}{UK-A} & \# of instance solved & & 20 & & 21 & & 27 \\
 &  Time/Gap & & 1965/50.0\% & & 1858/47.5\% & & 1462/15.8\% \\
\cline{2-8}
\multirow{2}{*}{UK-B} & \# of instance solved & & 32 & & 36 & & 32 \\
 &  Time/Gap & & 791/12.6\% & & 496/0.0 & & 808/0.1\% \\
\cline{2-8}
\multirow{2}{*}{UK-C} &  \# of instance solved & & 23 & & 26 & & 17 \\
 &  Time/Gap & & 1612/27.6\% & & 1300/6.1\% & & 2078/4.1\% \\
\hline
\end{tabular}
\caption{Comparison between the BCP algorithms with different route choices for the PRP instances}\label{prp_routechoice}
\end{table}

Our computational results show that which type of route to choose in our set partitioning formulation heavily depends on the characteristics of the instances. However, we observed some general rules. For example, q-route is not a good choice if some customers are clustered, which is the case for the maritime instances, where customers are clustered in Europe or Asia. This allows many q-routes with small cycles among the clustered customers, which makes it very time consuming for the labeling algorithm to finish.

\section{Conclusion} \label{sec:conclusion}
We studied a new transportation model that simultaneously optimizes the routing and speed decisions,
motivated by applications with reducing fuel consumption and GHG emissions as their objectives.
Our model can accommodate any empirical or physical fuel consumption function that is strictly convex in terms of speed. 
We then proposed a novel set partitioning formulation, where each column represents a combination of a routes a speed profile over that route. This formulation facilitates an efficient BCP algorithm with an effective dominance rule. Extensive computational results showed that the proposed BCP algorithm outperforms a state-of-the-art mixed-integer convex optimization solver.


%
\begin{appendices}
\section{Proof of Proposition~\ref{prop:dom}}
Suppose there exists $L^1$ satisfying conditions \ref{prop:dom:allowed}--\ref{prop:dom:speed}. 
Let $(P, v_P)\in E(L^2)$. For any $v_2\in S_v^2$ such that $(\bar{v}^2, v_2, v_P)$ is a feasible speed vector over $P^2\oplus P$, conditions \ref{prop:dom:lastvx}, \ref{prop:dom:allowed}, \ref{prop:dom:load}, and \eqref{eqn:dom:speed:time} ensure that there exists $v_1\in S_v^1$ such that $(\bar{v}^1, v_1, v_P)$ is a feasible speed vector over $P^1\oplus P$. Therefore, $(P,v_P) \in E(L^1)$.

We are going to show that for any triple $(P^2\oplus P, I^2, \vec{s}^2)$ induced by $v_2\in S_v^2$ that leads to a feasible speed vector $(\bar{v}^2, v_2, v_P)$ on $P^2\oplus P$, there exists a triple $(P^1\oplus P, I^3, \vec{s}^3)$ with a smaller reduced cost. It is sufficient to consider the optimal speed vector $\vec{v}^*$ over $\PL^2 \oplus P$ that is consistent with label $L^2$, i.e., the components of $\vec{v}^*$ from the depot to vertex $i_{w^2}$ are $\bar{v}^2$, and customers after $i_{w^2}$ and before $i^2$ (including $i^2$) are all seamless. Let $v^*_2$ be the component of $\vec{v}^*$ corresponding to the speed over the arc entering vertex $i^2$. Then the vehicle travels along $\PL^2 \oplus P$ in the following manner: start from the depot to vertex $i_{w^2}$ with speed $\bar{v}^2$, leave from vertex $i_{w^2}$ to vertex $i^2$ with speed $v^*_2$, and leave from vertex $i^2$ back to the depot with speeds $\vec{v}^*_P$, the components of $\vec{v}^*$ that correspond to $P$. Assume that the vehicle returns to the depot at time $t_{P^2 \oplus P}$, and the optimal cost on $P$ is $F^*$. Thus the total cost of route $P^2 \oplus P$ with speed vector $\vec{v}^*$ is $C^2(v_2) + F^*$.

Now we show how the route $P^1\oplus P$ admits a triple $(P^1\oplus P, I^3, \vec{s}^3)$ with a better reduced cost than $(P^2\oplus P, I^2, \vec{s})$ corresponding to the speed vector $\vec{v}^*$. Since $v_2 \in S_v^2$, by \eqref{eqn:dom:speed:time} there exists $v_1\in S_v^1$ such that $T^1(v_1) \le T^2(v_2)$.
Therefore, it is feasible to travel along $\PL^1 \oplus P$ in the following manner:
travel from the depot to vertex $i_{w^1}$ along $\PL^1$ with speed vector $\bar{v}^1$, travel from vertex $i_{w^1}$ to vertex $i^1$ with speed $v_1$, finish serving vertex $i^1$ at $T^1(v_1)$, leave vertex $i^1$ at time $T^2(v_2)$, travel along $P$ using speeds given by the corresponding components in $\vec{v}^*$, and return to the depot at time $t_{P^2\oplus P}$. Then the total cost of traveling along route $\PL^1\oplus P$ in the above manner is $C^1(v_1) + F^*$.

\noindent Figure \ref{graph-comparison} illustrates the comparison between $\PL^1 \oplus P$ and $\PL^2 \oplus P$.
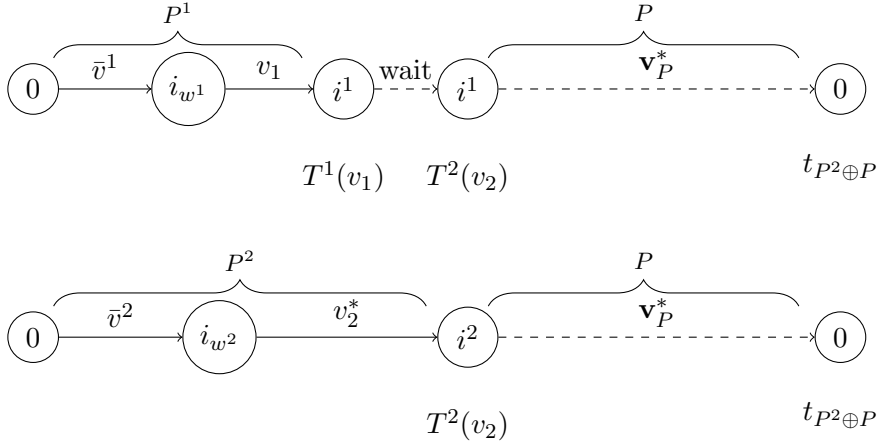
\begin{figure}[tbh]
\begin{center}
\begin{tikzpicture}[x=0.05\textwidth,y=0.05\textwidth]
\tikzstyle{every node}=[circle,draw,minimum
size=0.02\textwidth]


\node (a) at (-1,0) {$0$}; \node (b) at (2,0) {$i_{w^2}$}; \node (c) at (6,0) [label=below:$T^2(v_2)$] {$i^2$}; \node (d) at (12,0) [label=below:$t_{P^2 \oplus P}$] {$0$};

\node (a2) at (-1,4) {$0$}; \node (b2) at (1.5,4) {$i_{w^1}$}; \node (c2) at (4,4) [label=below:$T^1(v_1)$] {$i^1$}; \node (d2) at (6,4) [label=below:$T^2(v_2)$] {$i^1$}; \node (e2) at (12,4) [label=below:$t_{P^2 \oplus P}$] {$0$};

\tikzstyle{every node}=[]
\draw[->] (a) -- node [above] {$\bar{v}^2$}(b);
\draw[->] (b) -- node [above] {$v^*_2$}(c);
\draw[->,dashed] (c) -- node [above] {$\vec{v}^*_P$}(d);
\draw [decorate,decoration={brace,amplitude=10pt},xshift=-4pt,yshift=0pt] (-0.5,0.5) -- (5.5,0.5) node [black,midway,yshift=0.6cm] {\footnotesize $P^2$};

\draw [decorate,decoration={brace,amplitude=10pt},xshift=-4pt,yshift=0pt] (6.5,0.5) -- (11.5,0.5) node [black,midway,yshift=0.6cm] {\footnotesize $P$};

\draw[->] (a2) -- node [above] {$\bar{v}^1$}(b2);
\draw[->] (b2) -- node [above] {$v_1$}(c2);
\draw[->,dashed] (c2) -- node [above] {wait}(d2);
\draw[->,dashed] (d2) -- node [above] {$\vec{v}^*_P$}(e2);
\draw [decorate,decoration={brace,amplitude=10pt},xshift=-4pt,yshift=0pt] (-0.5,4.5) -- (3.5,4.5) node [black,midway,yshift=0.6cm] {\footnotesize $P^1$};

\draw [decorate,decoration={brace,amplitude=10pt},xshift=-4pt,yshift=0pt] (6.5,4.5) -- (11.5,4.5) node [black,midway,yshift=0.6cm] {\footnotesize $P$};
\end{tikzpicture}
\caption{An illustration of comparison between $\PL^1 \oplus P$ and $\PL^2 \oplus P$.}\label{graph-comparison}	
\end{center}
\end{figure}

Note that traveling along $\PL^1 \oplus P$ in the above manner does not necessarily respect the set of active customers in $L^1$. Nonetheless it is a feasible way to travel along $\PL^1 \oplus P$. Then there must exists a triple $(\PL^1\oplus P, I^3, \vec{s}^3)$ such that $c_{\PL^1 \oplus P, I^3, \vec{s}^3}$ is at most $C^1(v_1) + F^*$. Thus
\begin{subequations}
\begin{align*}
&\bar{c}_{\PL^1 \oplus P, I^3, \vec{s}^3} -\bar{c}_{P^2\oplus P,I^2, \vec{s}}
\le [C^1(v_1) + F^*  -( \mu( P^1) + \mu( P ) +\nu)] \\
-&[C^2(v_2) +  F^* -( \mu( P^2 ) + \mu( P ) +\nu)] \\
=& [C^1(v_1)  - \mu( P^1 ) ] -[C^2(v_2) - \mu( P^2 ) ]\\
< & 0.
\end{align*}
\end{subequations}
The last inequality follows from condition~\eqref{eqn:dom:speed:cost}. 

\section{Properties of functions $\beta$ and $H$}
The following proposition summarizes some properties of function $\beta$, where are easy to verify.
\begin{proposition}\label{observe} $\;$
\begin{itemize}
	\item For any $v\in [v^{\min}_2, v^{\max}_2]$ and $\beta(v) \in [v^{\min}_1, v^{\max}_1]$, $T^1(\beta(v))=T^2(v)$.
	\item Function $\beta$ is monotonically increasing.
	\item $T^1(v_1) \leq T^2(v_2)$ if and only if $\beta(v_2)\leq v_1$. In particular, $T^1(v^{\max}_1) \le T^2(v^{\max}_2)$ if and only if $\beta(v^{\max}_2) \le v^{\max}_1$.
\end{itemize}
\end{proposition}

\begin{lemma} \label{lem:H1}
When $v\in [l,u]$, the derivative of $H(v)$ is 0 if and only if $\beta(v)=v$. If $v^*$ is well-defined, then the derivative of $H(v)$ equals to 0 at $v=v^*$.
\end{lemma}
\begin{proof}
Since $H(v)= F^1_{\sd} + D^1 f(\beta(v)) - F^2_{\sd} - D^2f(v)$, its derivative
\begin{equation*}
\begin{split}
H'(v) &= D^1 f'(\beta(v)) \beta'(v) - D^2 f'(v)\\
&=D^1f'(\beta(v)) \frac{D^1D^2}{(D^2/v + \delta)^2v^2} - D^2 f'(v)\\
&=D^2f'(\beta(v)) (\frac{D^1}{D^2/v +\delta})^2 \frac{1}{v^2} - D^2 f'(v)\\
&=D^2 f'(\beta(v))\frac{(\beta(v))^2}{v^2} - D^2 f'(v).
\end{split}
\end{equation*}
Thus $H'(v)=0$ if and only if $f'(\beta(v))(\beta(v))^2 - G'(v)v^2=0$. The function $f'(v)v^2$ is strictly increasing for $v\in [v_F,u]$, since $f'(v)$ is non-negative and strictly increasing in $[v_F,u]$ and $v^2$ is non-negative and strictly increasing in $[0,\infty)$. When $v\in [v_F,u]$, $f'(\beta(v))(\beta(v))^2 - f'(v)v^2=0$ if and only if $\beta(v) = v$. When $v^*$ is well-defined, $\beta(v)=v$ if and only if $v=v^*$.
\end{proof}

\begin{lemma} \label{lem:H2}
Suppose $\delta = 0$. Then the maximum of $H(v)$ over an interval $[v^{\min}, v^{\max}]$ with $v^{\min}>0$ is attained at $v^{\max}$ if $D^1> D^2$, at $v^{\min}$ if $D^1< D^2$, and at any point in $[v^{\min}, v^{\max}]$ if $D^1=D^2$.
\end{lemma}
\begin{proof}
When $D^1= D^2$, $\beta(v)=v$ for any $v$ and $H(v)$ is constant. When $D^1> D^2$, $\beta(v)>v$ and $H'(v)>0$ for any $v\in [v^{\min}, v^{\max}]$, based on the proof of Lemma~\ref{lem:H1}. Thus $H(v)$ is strictly increasing and its maximum is attained at $v^{\max}$. The case of $D^1 < D^2$ can be proved in a similar way. 
\end{proof}

\section{Detailed computational results of the BCP algorithms with different route choices for the PRP instances}

\begin{table}[htb]
\footnotesize
\centering
\begin{tabular}{llllllllllllllll}
& & \multicolumn{4}{c}{Elementary} & & \multicolumn{4}{c}{2-cycle-free} & & \multicolumn{4}{c}{q-route} \\
\cline{3-6} \cline{8-11} \cline{13-16}
 & & \multicolumn{2}{c}{LP} & & & & \multicolumn{2}{c}{LP} & & & & \multicolumn{2}{c}{LP} \\
\cline{3-4} \cline{8-9} \cline{13-14}
 Instance & & Value & Time & BUB & Time & & Value & Time & BUB & Time & & Value & Time & BUB & Time \\
\hline
UK10A-1 & & 170.64 & 0.48 & 170.64 & 0.5 & & 170.64 & 1.43 & 170.64 & 1.4 & & 170.64 & 1.81 & 170.64 & 1.8 \\
UK10A-2 & & 204.88 & 0.26 & 204.88 & 0.3 & & 204.88 & 1.00 & 204.88 & 1.0 & & 204.88 & 1.18 & 204.88 & 1.2 \\
UK10A-3 & & 200.34 & 0.70 & 200.34 & 0.7 & & 200.34 & 0.88 & 200.34 & 0.9 & & 200.34 & 0.53 & 200.34 & 0.5 \\
UK10A-4 & & 189.88 & 0.22 & 189.88 & 0.2 & & 189.88 & 1.35 & 189.88 & 1.4 & & 189.88 & 1.40 & 189.88 & 3.5 \\
UK10A-5 & & 175.59 & 0.50 & 175.59 & 0.5 & & 175.59 & 3.40 & 175.59 & 3.4 & & 175.59 & 2.85 & 175.59 & 2.9 \\
UK10A-6 & & 214.48 & 0.24 & 214.48 & 0.2 & & 214.48 & 0.53 & 214.48 & 0.5 & & 214.48 & 0.78 & 214.48 & 0.8 \\
UK10A-7 & & 190.14 & 0.12 & 190.14 & 0.1 & & 190.14 & 1.74 & 190.14 & 1.7 & & 190.14 & 2.81 & 190.14 & 2.8 \\
UK10A-8 & & 222.17 & 0.03 & 222.17 & 0.0 & & 222.17 & 0.15 & 222.17 & 0.1 & & 222.17 & 0.14 & 222.17 & 0.1 \\
UK10A-9 & & 174.54 & 0.12 & 174.54 & 0.1 & & 174.54 & 2.07 & 174.54 & 2.1 & & 174.54 & 6.69 & 174.54 & 6.7 \\
UK10A-10 & & 189.82 & 0.39 & 189.82 & 0.4 & & 189.82 & 0.42 & 189.82 & 0.4 & & 189.82 & 0.59 & 189.82 & 0.6 \\
\hline
UK10B-1 & & 246.44 & 0.01 & 246.44 & 0.0 & & 246.44 & 0.01 & 246.44 & 0.0 & & 246.44 & 0.01 & 246.44 & 0.0 \\
UK10B-2 & & 303.73 & 0.01 & 303.73 & 0.0 & & 303.73 & 0.01 & 303.73 & 0.0 & & 303.73 & 0.01 & 303.73 & 0.0 \\
UK10B-3 & & 301.89 & 0.01 & 301.89 & 0.0 & & 301.89 & 0.01 & 301.89 & 0.0 & & 301.89 & 0.01 & 301.89 & 0.0 \\
UK10B-4 & & 273.90 & 0.01 & 273.90 & 0.0 & & 273.90 & 0.01 & 273.90 & 0.0 & & 273.90 & 0.02 & 273.90 & 0.0 \\
UK10B-5 & & 255.07 & 0.01 & 255.07 & 0.0 & & 255.47 & 0.00 & 255.10 & 0.0 & & 255.47 & 0.00 & 255.07 & 0.0 \\
UK10B-6 & & 332.34 & 0.01 & 332.34 & 0.0 & & 332.34 & 0.01 & 332.34 & 0.0 & & 332.34 & 0.01 & 332.34 & 0.0 \\
UK10B-7 & & 308.85 & 0.00 & 314.64 & 0.1 & & 308.85 & 0.00 & 314.64 & 0.1 & & 304.91 & 0.00 & 314.64 & 0.1 \\
UK10B-8 & & 333.12 & 0.00 & 339.32 & 0.0 & & 333.11 & 0.00 & 339.32 & 0.0 & & 333.12 & 0.00 & 339.32 & 0.0 \\
UK10B-9 & & 261.10 & 0.01 & 261.10 & 0.0 & & 261.10 & 0.00 & 261.10 & 0.0 & & 261.10 & 0.01 & 261.10 & 0.0 \\
UK10B-10 & & 285.20 & 0.01 & 285.20 & 0.0 & & 285.20 & 0.01 & 285.20 & 0.0 & & 285.20 & 0.00 & 285.20 & 0.0 \\
\hline
UK10C-1 & & 210.21 & 0.02 & 210.21 & 0.0 & & 210.18 & 0.05 & 210.18 & 0.1 & & 209.66 & 0.10 & 210.18 & 0.2 \\
UK10C-2 & & 271.93 & 0.03 & 271.93 & 0.0 & & 268.93 & 0.10 & 271.93 & 0.3 & & 265.77 & 0.20 & 271.93 & 1.1 \\
UK10C-3 & & 229.18 & 0.03 & 229.18 & 0.0 & & 229.18 & 0.05 & 229.18 & 0.1 & & 227.41 & 0.20 & 229.18 & 0.5 \\
UK10C-4 & & 230.52 & 0.01 & 230.52 & 0.0 & & 227.31 & 0.00 & 230.52 & 0.2 & & 222.99 & 0.10 & 230.52 & 6.2 \\
UK10C-5 & & 205.49 & 0.02 & 205.49 & 0.0 & & 205.49 & 0.08 & 205.49 & 0.1 & & 203.33 & 0.30 & 205.49 & 0.8 \\
UK10C-6 & & 255.82 & 0.02 & 255.82 & 0.0 & & 255.82 & 0.03 & 255.82 & 0.0 & & 254.48 & 0.10 & 255.82 & 0.1 \\
UK10C-7 & & 217.79 & 0.03 & 217.79 & 0.0 & & 217.79 & 0.07 & 217.79 & 0.1 & & 217.79 & 0.33 & 217.79 & 0.3 \\
UK10C-8 & & 251.29 & 0.03 & 251.29 & 0.0 & & 251.29 & 0.02 & 251.29 & 0.0 & & 251.29 & 0.10 & 251.29 & 0.1 \\
UK10C-9 & & 186.04 & 0.06 & 186.04 & 0.1 & & 186.04 & 0.12 & 186.04 & 0.1 & & 186.04 & 0.09 & 186.04 & 0.1 \\
UK10C-10 & & 231.62 & 0.04 & 231.62 & 0.0 & & 231.62 & 0.02 & 231.62 & 0.0 & & 230.93 & 0.10 & 231.62 & 0.2 \\
\hline
\end{tabular}
\caption{Results of the BCP algorithms under different route choices for the PRP instances with 10 customers} \label{table-prp10}
\end{table}
\begin{table}[htb]
\footnotesize
\centering
\begin{tabular}{llllllllllllllll}
& & \multicolumn{4}{c}{elementary} & & \multicolumn{4}{c}{2-cycle-free} & & \multicolumn{4}{c}{q-route} \\
\cline{3-6} \cline{8-11} \cline{13-16}
 & & \multicolumn{2}{c}{LP} & & & & \multicolumn{2}{c}{LP} & & & & \multicolumn{2}{c}{LP} \\
\cline{3-4} \cline{8-9} \cline{13-14}
 Instance & & Value & Time & BUB & Time & & Value & Time & BUB & Time & & Value & Time & BUB & \begin{tabular}{l} Time/\\Gap \end{tabular}\\
\hline
UK20A-1 & & 351.82 & 26.5 & 351.82 & 26.5 & & 351.69 & 41.6 & 351.82 & 64.0 & & 350.46 & 14.5 & 351.82 & 19.4 \\
UK20A-2 & & 365.77 & 31.5 & 365.77 & 31.5 & & 365.77 & 16.7 & 365.77 & 16.7 & & 365.77 & 4.68 & 365.77 & 4.7 \\
UK20A-3 & & 230.49 & 1204 & 230.49 & 1204 & & 230.49 & 556 & 230.49 & 556 & & 230.49 & 52.9 & 230.49 & 52.9 \\
UK20A-4 & & 347.04 & 2370 & 347.04 & 2370 & & 347.04 & 659 & 347.04 & 659 & & 347.04 & 59.4 & 347.04 & 59.4 \\
UK20A-5 & & 323.44 & 154 & 323.44 & 154 & & 323.44 & 134 & 323.44 & 134 & & 323.31 & 17.3 & 323.44 & 28.7 \\
UK20A-6 & & 364.11 & 437 & 364.23 & 757 & & 364.11 & 40.5 & 364.23 & 94.6 & & 363.67 & 6.00 & 364.23 & 17.0 \\
UK20A-7 & & - & - & 3318.70 & - & & - & - & 3016.70 & - & & 253.43 & 2429 & 253.61 & 0.1\% \\
UK20A-8 & & 301.51 & 151 & 301.51 & 151 & & 301.51 & 50.1 & 301.51 & 50.1 & & 301.51 & 15.9 & 301.51 & 15.9 \\
UK20A-9 & & 362.56 & 38.6 & 362.56 & 38.6 & & 362.56 & 98.3 & 362.56 & 98.3 & & 362.56 & 11.4 & 362.56 & 11.4 \\
UK20A-10 & & 313.33 & 455 & 313.33 & 455 & & 313.33 & 231 & 313.33 & 231 & & 313.33 & 16.7 & 313.33 & 16.7 \\
\hline
UK20B-1 & & 469.35 & 0.04 & 469.35 & 0.0 & & 469.35 & 0.04 & 469.35 & 0.0 & & 469.35 & 0.04 & 469.35 & 0.0 \\
UK20B-2 & & 477.05 & 0.18 & 477.05 & 0.2 & & 477.05 & 0.06 & 477.05 & 0.1 & & 477.05 & 0.10 & 477.05 & 0.1 \\
UK20B-3 & & 354.46 & 0.11 & 354.46 & 0.1 & & 354.46 & 0.11 & 354.46 & 0.1 & & 354.46 & 0.11 & 354.46 & 0.1 \\
UK20B-4 & & 523.59 & 0.12 & 523.59 & 0.1 & & 523.59 & 0.08 & 523.59 & 0.1 & & 523.59 & 0.15 & 523.59 & 0.1 \\
UK20B-5 & & 447.33 & 0.08 & 447.33 & 0.1 & & 447.33 & 0.07 & 447.33 & 0.1 & & 439.38 & 0.20 & 447.33 & 1.2 \\
UK20B-6 & & 511.11 & 0.40 & 511.78 & 1.4 & & 511.11 & 0.10 & 511.78 & 0.4 & & 511.78 & 1.34 & 511.78 & 1.3 \\
UK20B-7 & & 377.94 & 1.60 & 379.02 & 4.0 & & 379.02 & 0.62 & 379.02 & 0.6 & & 377.57 & 0.30 & 379.02 & 0.4 \\
UK20B-8 & & 418.00 & 0.10 & 431.31 & 0.3 & & 418.00 & 0.10 & 431.31 & 0.6 & & 416.64 & 0.10 & 431.31 & 1.7 \\
UK20B-9 & & 545.37 & 0.10 & 548.68 & 0.1 & & 545.37 & 0.10 & 548.68 & 0.1 & & 548.68 & 0.23 & 548.68 & 0.2 \\
UK20B-10 & & 410.32 & 0.04 & 410.32 & 0.0 & & 410.32 & 0.03 & 410.32 & 0.0 & & 407.96 & 0.10 & 410.32 & 0.2 \\
\hline
UK20C-1 & & 432.35 & 0.73 & 432.35 & 0.7 & & 424.06 & 1.30 & 432.35 & 90.3 & & 420.38 & 1.60 & 432.35 & 0.9\% \\
UK20C-2 & & 445.64 & 2.00 & 448.29 & 12.9 & & 444.96 & 1.30 & 448.29 & 4.4 & & 442.07 & 1.00 & 448.29 & 7.7 \\
UK20C-3 & & 287.48 & 4.39 & 287.48 & 4.4 & & 287.20 & 2.90 & 287.48 & 4.1 & & 281.39 & 3.40 & 287.04 & 76.2 \\
UK20C-4 & & 432.55 & 0.90 & 434.23 & 2.9 & & 431.69 & 0.90 & 434.23 & 4.4 & & 422.25 & 1.5 & 434.23 & 0.7\% \\
UK20C-5 & & 380.28 & 1.70 & 381.70 & 6.1 & & 379.85 & 1.40 & 381.70 & 4.5 & & 375.64 & 1.60 & 381.70 & 41.9 \\
UK20C-6 & & 441.43 & 2.50 & 444.35 & 23.4 & & 439.28 & 1.10 & 444.35 & 28.9 & & 438.56 & 0.80 & 444.35 & 16.3 \\
UK20C-7 & & 317.73 & 800 & 317.73 & 800 & & 317.44 & 52.0 & 317.73 & 63.1 & & 308.47 & 34.8 & 317.73 & 0.9\% \\
UK20C-8 & & 410.35 & 0.57 & 410.35 & 0.6 & & 410.35 & 0.52 & 410.35 & 0.5 & & 410.16 & 1.00 & 410.35 & 1.4 \\
UK20C-9 & & 421.39 & 0.25 & 421.39 & 0.2 & & 421.39 & 0.49 & 421.39 & 0.5 & & 410.48 & 1.30 & 421.39 & 44.3 \\
UK20C-10 & & 384.88 & 0.54 & 384.88 & 0.5 & & 380.30 & 1.20 & 384.88 & 13.7 & & 373.95 & 1.90 & 384.88 & 130 \\
\hline
\end{tabular}
\caption{Results of the BCP algorithms under different route choices for the PRP instances with 20 customers}
\end{table}

\begin{landscape}
\begin{table}[htb]
\small
\centering
\renewcommand{\tabcolsep}{2.0mm}
\begin{tabular}{llllllllllllllll}
& & \multicolumn{4}{c}{Elementary} & & \multicolumn{4}{c}{2-cycle-free} & & \multicolumn{4}{c}{q-route} \\
\cline{3-6} \cline{8-11} \cline{13-16}
 & & \multicolumn{2}{c}{LP} & & & & \multicolumn{2}{c}{LP} & & & & \multicolumn{2}{c}{LP} \\
\cline{3-4} \cline{8-9} \cline{13-14}
 Instance & & Value & Time & BUB & \begin{tabular}{l} Time/\\Gap \end{tabular} & & Value & Time & BUB & \begin{tabular}{l} Time/\\Gap \end{tabular} & & Value & Time & BUB & \begin{tabular}{l} Time/\\Gap \end{tabular} \\
\hline
UK25A-1 & & - & - & 3316.14 & - & & - & - & 389.46 & - & & 316.18 & 304 & 316.18 & 304 \\
UK25A-2 & & - & - & 397.10 & - & & - & - & 373.92 & - & & 373.92 & 1244 & 373.92 & 1562 \\
UK25A-3 & & - & - & 268.46 & - & & - & - & 273.08 & - & & - & - & 3591.44 & - \\
UK25A-4 & & - & - & 677.11 & - & & - & - & 321.14 & - & & - & - & 880.80 & - \\
UK25A-5 & & - & - & 3932.00 & - & & 365.00 & 2925 & 365.00 & 2925 & & 365.00 & 169 & 365.00 & 169 \\
UK25A-6 & & - & - & 397.30 & - & & - & - & 313.72 & - & & 312.22 & 552 & 316.83 & 1.2\% \\
UK25A-7 & & - & - & 467.97 & - & & - & - & 3431.12 & - & & 350.72 & 610 & 350.72 & 610 \\
UK25A-8 & & 359.08 & 1432 & 359.08 & 1432 & & 359.08 & 1107 & 359.08 & 1107 & & 358.69 & 84.1 & 359.08 & 118 \\
UK25A-9 & & - & - & 346.40 & - & & - & - & 380.68 & - & & 325.91 & 184 & 330.60 & 0.9\% \\
UK25A-10 & & - & - & 411.33 & - & & - & - & 411.33 & - & & 411.33 & 3503 & 411.33 & 3503 \\
\hline
UK25B-1 & & 474.35 & 2.20 & 475.13 & 6.5 & & 474.35 & 0.60 & 475.13 & 1.0 & & 473.63 & 0.50 & 475.13 & 1.3 \\
UK25B-2 & & 530.00 & 8.60 & 533.59 & 141 & & 530.00 & 0.50 & 533.59 & 10.6 & & 533.59 & 63.6 & 533.59 & 63.6 \\
UK25B-3 & & 390.17 & 15.7 & 390.27 & 37.9 & & 390.17 & 1.80 & 390.27 & 2.6 & & 388.45 & 1.90 & 390.27 & 9.1 \\
UK25B-4 & & 467.49 & 239 & 467.49 & 239 & & 458.56 & 0.80 & 467.49 & 132 & & 449.56 & 0.80 & 467.49 & 1.4\% \\
UK25B-5 & & 489.16 & 192 & 491.69 & 1286 & & 489.16 & 2.50 & 491.69 & 11.2 & & 488.80 & 1.60 & 491.69 & 10.0 \\
UK25B-6 & & 515.46 & 1.00 & 516.71 & 16.0 & & 515.46 & 0.30 & 516.71 & 1.5 & & 514.11 & 0.40 & 516.71 & 6.5 \\
UK25B-7 & & 525.68 & 120 & 526.12 & 164 & & 524.90 & 0.70 & 526.12 & 2.3 & & 525.69 & 0.40 & 526.12 & 2.0 \\
UK25B-8 & & 534.38 & 1.60 & 537.93 & 126 & & 533.59 & 0.40 & 537.15 & 3.5 & & 533.38 & 0.50 & 537.15 & 3.1 \\
UK25B-9 & & 436.23 & 218 & 436.84 & 464 & & 436.23 & 1.70 & 436.84 & 2.8 & & 434.89 & 1.40 & 436.84 & 3.5 \\
UK25B-10 & & 531.86 & 13.5 & 533.88 & 126 & & 531.86 & 1.00 & 533.88 & 5.0 & & 531.10 & 0.90 & 533.88 & 3.4 \\
\hline
UK25C-1 & & 417.59 & 1826 & 417.59 & 1826 & & 416.88 & 52.9 & 417.59 & 163 & & 409.86 & 61.9 & 417.59 & 0.7\% \\
UK25C-2 & & 467.87 & 1506 & 468.95 & 0.1\% & & 467.51 & 53.7 & 468.95 & 548 & & 453.86 & 25.3 & 481.59 & 4.0\% \\
UK25C-3 & & - & - & 425.39 & - & & 326.31 & 329 & 328.47 & 0.1\% & & 323.41 & 55.0 & 329.34 & 1.1\% \\
UK25C-4 & & 388.98 & 48.4 & 389.56 & 77.8 & & 386.21 & 3.60 & 391.34 & 1.3\% & & 367.54 & 33.4 & 391.11 & 4.4\% \\
UK25C-5 & & - & - & 1444.90 & - & & - & - & 458.55 & - & & - & - & 456.58 & - \\
UK25C-6 & & 415.31 & 436 & 418.12 & 0.3\% & & 415.29 & 8.80 & 418.12 & 135 & & 412.09 & 5.30 & 418.12 & 0.4\% \\
UK25C-7 & & - & - & 727.00 & - & & 478.25 & 48.1 & 482.31 & 0.4\% & & 473.62 & 13.8 & 481.75 & 1.2\% \\
UK25C-8 & & 497.49 & 158 & 499.18 & 560 & & 496.31 & 8.00 & 499.18 & 30.5 & & 481.75 & 4.40 & 499.18 & 1.1\% \\
UK25C-9 & & 406.22 & 109 & 413.28 & 1.0\% & & 405.99 & 7.50 & 412.51 & 520 & & 400.25 & 6.20 & 412.51 & 1.1\% \\
UK25C-10 & & 509.00 & 1441 & 551.17 & 8.3\% & & 508.78 & 40.6 & 513.37 & 0.4\% & & 497.67 & 4.80 & 513.47 & 1.5\% \\
\hline
\end{tabular}
\caption{Results of the BCP algorithms under different route choices for the PRP instances with 25 customers}
\end{table}
\begin{table}
\small
\centering
\renewcommand{\tabcolsep}{1.50mm}
\begin{tabular}{llllllllllllllll}
& & \multicolumn{4}{c}{Elementary} & & \multicolumn{4}{c}{2-cycle-free} & & \multicolumn{4}{c}{q-routes} \\
\cline{3-6} \cline{8-11} \cline{13-16}
 & & \multicolumn{2}{c}{LP} & & & & \multicolumn{2}{c}{LP} & & & & \multicolumn{2}{c}{LP} \\
\cline{3-4} \cline{8-9} \cline{13-14}
 Instance & & Value & Time & BUB & \begin{tabular}{l} Time/\\Gap \end{tabular} & & Value & Time & BUB & \begin{tabular}{l} Time/\\Gap \end{tabular} & & Value & Time & BUB & \begin{tabular}{l} Time/\\Gap \end{tabular} \\
\hline
UK50A-1 & & - & - & 2748.2 & - & & - & - & 689.6 & - & & 673.3 & 1884 & 6425.2 & 846\% \\
UK50A-2 & & - & - & 3119.9 & - & & - & - & 929.3 & - & & 668.5 & 3583 & 668.5 & 3583 \\
UK50A-3 & & - & - & 7512.3 & - & & - & - & 7750.3 & - & & 688.2 & 544 & 988.6 & 43\% \\
UK50A-4 & & - & - & 3653.8 & - & & - & - & 4085.5 & - & & - & - & 3845.7 & - \\
UK50A-5 & & - & - & 2847.0 & - & & - & - & 7434.9 & - & & 702.2 & 421 & 706.3 & 0.4\% \\
UK50A-6 & & - & - & 3464.4 & - & & - & - & 6697.3 & - & & 631.3 & 1621 & 631.3 & 1621 \\
UK50A-7 & & - & - & 833.9 & - & & - & - & 7774.5 & - & & - & - & 6758.5 & - \\
UK50A-8 & & - & - & 650.6 & - & & - & - & 638.9 & - & & 636.9 & 1052 & 638.5 & 0.1\% \\
UK50A-9 & & - & - & 7292.9 & - & & - & - & 3728.1 & - & & - & - & 786.0 & - \\
UK50A-10 & & - & - & 2182.7 & - & & - & - & 7369.2 & - & & - & - & 6533.2 & - \\
\hline 
UK50B-1 & & - & - & 3055.38 & - & & 984.90 & 6.80 & 992.15 & 0.3\% & & 978.67 & 4.70 & 992.00 & 0.7\% \\
UK50B-2 & & - & - & 1887.07 & - & & 1021.91 & 5.90 & 1028.01 & 0.0\% & & 1018.19 & 4.80 & 1028.01 & 0.4\% \\
UK50B-3 & & 1013.47 & 1807 & 8857.58 & 774.0\% & & 1013.44 & 8.40 & 1022.40 & 0.5\% & & 1009.38 & 10.1 & 1022.69 & 0.8\% \\
UK50B-4 & & 1141.88 & 596 & 2624.56 & 129.8\% & & 1141.85 & 5.50 & 1149.08 & 1985 & & 1140.73 & 3.30 & 1149.08 & 3408 \\
UK50B-5 & & 963.84 & 61.6 & 964.72 & 185 & & 963.84 & 4.00 & 964.72 & 6.8 & & 960.57 & 1.70 & 964.72 & 14.8 \\
UK50B-6 & & - & - & 8962.33 & - & & 911.64 & 14.5 & 921.95 & 0.8\% & & 907.76 & 7.50 & 924.94 & 1.4\% \\
UK50B-7 & & - & - & 3407.17 & - & & 847.84 & 16.2 & 852.33 & 1022 & & 839.92 & 8.20 & 852.33 & 0.8\% \\
UK50B-8 & & 924.43 & 50.5 & 924.43 & 50.5 & & 924.43 & 3.62 & 924.43 & 3.6 & & 923.90 & 6.20 & 924.43 & 13.8 \\
UK50B-9 & & 1027.56 & 2472 & 1592.42 & 55.0\% & & 1027.41 & 3.40 & 1034.92 & 2252 & & 1023.36 & 2.50 & 1036.14 & 0.4\% \\
UK50B-10 & & - & - & 8504.07 & - & & 1022.60 & 83.2 & 1026.80 & 0.2\% & & 1017.29 & 17.9 & 1033.92 & 1.2\% \\
\hline
UK50C-1 & & - & - & 8400.51 & - & & 900.91 & 308 & 1198.26 & 32.8\% & & 886.64 & 57.3 & 1243.50 & 38.8\% \\
UK50C-2 & & - & - & 1846.43 & - & & 884.84 & 771 & 1227.79 & 38.6\% & & 872.69 & 73.1 & 1387.62 & 58.2\% \\
UK50C-3 & & - & - & 4949.76 & - & & 851.70 & 680 & 1015.20 & 19.0\% & & 844.60 & 65.8 & 869.18 & 2.5\% \\
UK50C-4 & & - & - & 7966.25 & - & & 996.01 & 200 & 1003.43 & 0.5\% & & 990.32 & 38.5 & 1002.20 & 0.8\% \\
UK50C-5 & & 916.44 & 2018 & 8394.03 & 815.9\% & & 914.41 & 48.6 & 925.85 & 1.0\% & & 902.56 & 22.1 & 1177.71 & 29.5\% \\
UK50C-6 & & - & - & 7557.40 & - & & 792.78 & 1033 & 2607.89 & 228.5\% & & 784.36 & 50.6 & 1569.29 & 99.4\% \\
UK50C-7 & & - & - & 1754.22 & - & & 851.59 & 508 & 852.67 & 0.0\% & & 834.81 & 64.0 & 895.22 & 6.4\% \\
UK50C-8 & & - & - & 826.66 & - & & 789.03 & 463 & 1233.17 & 56.0\% & & 782.84 & 87.6 & 2643.45 & 236.6\% \\
UK50C-9 & & 959.55 & 331 & 962.69 & 0.2\% & & 957.10 & 32.5 & 962.69 & 0.2\% & & 951.46 & 12.2 & 963.40 & 0.8\% \\
UK50C-10 & & - & - & 4547.35 & - & & - & - & 8330.19 & - & & 959.03 & 152 & 985.50 & 2.5\% \\
\hline
\end{tabular}
\caption{Results of the BCP algorithms under different route choices for the PRP instances with 50 customers} \label{table-prp50}
\end{table}
\end{landscape}

\end{appendices}


\clearpage
\bibliographystyle{informs2014} 

\bibliography{JRSP}


\end{document}